\RequirePackage{silence}
\documentclass[a4paper]{article}
\usepackage{latexstyle}
\date{\today}
\title{The Navier-Stokes-Vlasov-Fokker-Planck system as a scaling limit of particles in a fluid}
\author{Franco Flandoli\footnote{franco.flandoli@sns.it. Scuola Normale Superiore of Pisa, Italy.}, Marta Leocata\footnote{leocata@mail.dm.unipi.it. University of Pisa, Italy.}, Cristiano Ricci\footnote{cristiano.ricci@unifi.it. University of Florence, Italy.}}

\begin{document}
\maketitle 
\begin{abstract}
Convergence of a system of particles, interacting with a fluid, to Navier-Stokes-Vlasov-Fokker-Planck system is studied. The interaction between
particles and fluid is described by Stokes drag force. The empirical
measure of particles is proved to converge to the Vlasov-Fokker-Planck component of the system and the velocity of the fluid coupled with the particles converges in the uniform topology to the the Navier-Stokes component. A new uniqueness result for the PDE system is added.
\end{abstract}
% \tableofcontents
\section{Introduction}\label{sec:intro}
In the theory of multiphase flows, the coupled PDE system called Navier-Stokes-Vlasov-Fokker-Planck is a way of modeling the behavior of a large number of particles immersed into a fluid. It is made by two major components: a vector field $u$, representing the velocity of the fluid at a given time and position, and a scalar valued function $F$, representing the density on phase space of the particles immersed in the fluid. In the incompressible case, when the interaction between particles and fluid is modelled by Stokes drag force, the system is given by the following equations
\begin{equation}\label{eq:PDEVNSIntro}
\begin{cases}
\partial_{t} u = \Delta u -u \cdot \nabla u -\nabla \pi -\int
_{\RR^{d}}(u-v)F\,dv;\\
\div(u)=0;\\
\partial_{t}F + v \cdot \nabla_{x}F +\div_{v}((u-v)F) = \frac{\sigma^{2}}{2}\Delta_{v}F.
\end{cases}
\end{equation}
% Often the case $\sigma =0$ is considered. Here we deal with the case $\sigma > 0$ because of technical reasons; this case is sometimes more properly called Vlasov-Fokker-Planck-Navier-Stokes system but we use the shorter name Vlasov-Navier-Stokes system for simplicity. 
Often the case $\sigma= 0$ is considered in the literature. Here we deal with the case $\sigma > 0$ because of technical reasons. The case $\sigma =0$ is usually called Vlasov-Navier-Stokes (VNS); the case $\sigma > 0$, Navier-Stokes-Vlasov-Fokker-Planck. In the sequel, for simplicity of notations, we will often call VNS also the system above with $\sigma > 0$. 

The PDE description for the density of particles is reasonable when the number of particles is very large and overcomes the problem of describing the details of each single particle. The aim of this paper is to prove that this simplification is correct: we prove that a system composed by Newtonian particles and fluid converges to the PDE system when the number of particles tends to infinity. 

The mathematical analysis of the coupled system \eqref{eq:PDEVNSIntro} in dimension $d = 2,3$ has received
much attention in the past years. Earlier result of global existence of weak solutions and large asymptotic for Stokes-Vlasov system in a bounded domain appeared in \cite{Hamdache1998}. Existence of weak solutions has been extend to the Navier-Stokes case, hence including the convection term in the equation for the fluid, in a periodic domain in \cite{boudin2009global}. Global existence of smooth solutions
with small data for Navier-Stokes-Vlasov-Fokker-Planck was obtained first in \cite{goudon2010navier}. In \cite{CYu} global existence for smooth solutions is generalized for large data.  Recent result on the topic of uniqueness have been obtained in the case $\sigma = 0$  in \cite{han2017uniqueness}. We shall prove a variant of these results adapted to the regularity of our solutions. Uniqueness plays a fundamental role in the mathematical problem we are interested in; existence is less relevant because it is obtained as a byproduct of our convergence result.

As said above, the aim of this work is to investigate a coupling between the fluid and a particle system, which converges, in the limit of large number of particles, to system  \eqref{eq:PDEVNSIntro}.
The literature on this topic is still fragmentary. The works \cite{goudon2004light} \cite{goudon2004fine}, present results of PDE to PDE convergence, only implicitly motivated by particle arguments. The works  \cite{allaire1991homogenizationI}, \cite{allaire1991homogenizationII}, \cite{golse2016aerosol}, \cite{golse2008meanfield}, \cite{mathiaud2010}, \cite{feireisl2016homogenization}, \cite{otto2004} aim to treat links between particles and fluid but, in the trade-off between different levels of mathematical complexity and physical realism, they choose to include the correct boundary conditions for the interaction between finite size particles and fluid and thus, due to severe technical difficulties, have to restrict to simplified fluid regimes. Our choice here, compared to these works is to consider a sort of phenomenological description of the interaction between particles and fluid (keeping the structure of Stokes drag force), but consider the usual Navier-Stokes regime, devoting the attention to other technical problems related to the macroscopic limit, instead of the very difficult problem of the precise boundary conditions between particles and fluid. 
The system considered here has the form

\begin{equation*}
\begin{cases}
\frac{\partial u^N}{\partial t}=\Delta u^N-u^N\cdot \nabla u^N-\nabla \pi^N-\frac{1}{N}\sum_{i=1}^{N} \left(u^{N}_{\eps_{N}}(t,X^{i,N}_{t})-V^{i}_{t}\right)\delta^{\eps_{N}}_{X^{i,N}_{t}},\\
\div(u^{N})=0,\\

\begin{cases}
dX^{i,N}_{t}=V^{i,N}_{t}dt,\\
d V^{{i,N}}_{t}= \left(u^{N}_{\eps_{N}}\left(t,X^{i,N}_{t}\right)-V^{i,N}_{t}\right)dt+\sigma dB^{i}_{t}
\end{cases}
\end{cases}
\end{equation*}
where $N$ is the number of particles and ($X_{t}^{i,N},V_{t}^{i,N})$ are position and velocity of the particles.
The equations for the fluid velocity and pressure $(u^{N},\pi^{N})$ are the classical Navier-Stokes equations for an incompressible Newtonian fluid, but now (compared to the PDE system above) with an interaction with particles of discrete type. We choose a phenomenological description of the interaction:\\
i) concerning the intensity of the force exerted by the fluid on each single particle and viceversa (by Newton's third law, the same term with opposite sign appears in the Navier-Stokes equation and in the ordinary differential equation for the particle velocity), it is given by the difference between the particle velocity and a local average of fluid velocity around particle position 
% we  compute the difference between the particle velocity $V^{i,N}_t$ and a local average at particle center $X^{i,N}_t$ of the fluid velocity $u^N_{\eps_N}$
\[
u^{N}_{\eps_{N}}(t,X^{i,N}_{t})=(\theta^{0,\eps_{N}}*u^{N}_t)(X^{i,N}_{t});
\]
ii) concerning the mechanism of action of a particle on the fluid, similarly to point i) we impose an action distributed in a small neighbor of particle position, as described by the mollified delta Dirac function
% where $\theta^{0,\eps_{N}}$ is a sequence of rescaled mollifier. Particles act on the fluid as a sort of Dirac forces, with intensity dependent on the relative velocity between fluid and particle. For technical reasons, but also as a trace of the fact that particles occupy a volume, we use a smoothed version of Dirac to describe the force
\[
\delta^{\eps_{N}}_{X^{i,N}_{t}}(x)=\theta^{0,\eps_{N}}(x-X^{i,N}_{t}).
\]
% where particles act on the fluid in a way that mimics the well known \emph{Stokes drag force}
% \[
% \int_{\RR^{d}}(u-v)F\,dv
% \]
% present in \eqref{eq:PDEVNSIntro}. Similarly, the shape of the interaction with the fluid used in the equations for particle velocity, resemble the transport term 
% \[
% \div_{v}((u-v)F)
% \]
% of the Vlasov part of system \eqref{eq:PDEVNSIntro}.
The choice to use local averages and locally distributed action is obviously an artefact compared to reality, convenient for the mathematical investigation; still it preserves the idea that particles are not just points but finite objects, or at least objects with a finite action radius, a sort of small boundary layer of interaction with the fluid.

Finally, let us comment on our previous works \cite{flandoli2016fluid}, \cite{oldVNS}. They both deal with a similar particle system coupled with the fluid and the question of its scaling limit. However, they are affected by important restrictions. The paper \cite{flandoli2016fluid} discusses only the so called two step approach. In this setting one keeps $\eps$ fixed when $N\rightarrow\infty$ and removes $\epsilon$ only later, as a second step. As usual, the analysis of such disjoint limits is much simpler, the first one being a classical mean field problem (opposite to the problem considered here, see the next section on the technical difficulties), the second one being a question of convergence of PDEs to PDEs, essentially a repetition of schemes known from the proofs of existence theorems for the limit system. One can mix the parameters a posteriori, taking subsequences, but the conditions on the link are quite unrealistic and restrictive. The paper \cite{oldVNS} on the contrary treats the joint limit in the two parameters, as in the present work but, at that time, we did not identify certain estimates and arguments, so the result in \cite{oldVNS} requires a special bounded modification of Stokes law and proves only convergence of subsequences, due to lack of a suitable uniqueness result. Compared to \cite{flandoli2016fluid}, \cite{oldVNS}, the result proved here is complete, without the main restrictions of those works. For future research, however, it would be interesting to extend further the range of the parameter $\beta$ that quantifies the radius of interaction between a particle and the surrounding fluid, see below and in the same vein, but more important, to treat more realistic boundary conditions between particles and fluid.

% In the direction of the phenomenological description we  recall a work based on the so-called two step approach, see \cite{flandoli2016fluid}. In this setting one keeps $\epsilon$ fixed when $N\rightarrow\infty$ and removes
% $\epsilon$ only later, as a second step; in this way, it is sufficient to have
% a good regularity property for $u_{\epsilon}^{N}$ to p in pass to the limit in the
% previous interaction terms, similarly to every mean-field problem. But in the
% realistic scaling limit problem studied here, where the smoothing is removed
% together with $N$, we need uniform convergence of $u_{\epsilon_{N}}^{N}$ to
% $u$. 
% We will prove rigorously that in dimension two, under suitable assumptions, the particle system presented above converges in the limit of large numbers and in the appropriate sense to system of equation \eqref{eq:PDEVNSIntro}. 
% With respect to the previous work we managed to overcome all the technical difficulties which are presented below.

\subsection{Difficulties}
In this subsection we aim to highlight the difficulties we met in proving the convergence from the discrete to the continuous model. Apparently it looks a mean field result but several aspects are far from standard, as we now describe.
\subsubsection{Uniform control on velocity and vorticity creation by particles}
The rough structure of the particle approximation used here is of a mean field
type. The empirical measure $S_{t}^{N}$ of the particles
\[
S^{N}_{t}=\frac{1}{N} \sum_{i=1}^{N} \delta_{(X^{i,N}_{t},V^{i,N}_{t})}
\] 
(see also Section \ref{sec:notationanmain}) will
be proved to converge to the solution $F_{t}\left(  x,v\right)  $ of the
Vlasov component of our system (in parallel, the approximation of the velocity
field will converge to the limit velocity field). In classical mean field
problems, however, after it is proved that $S_{t}^{N}$ converges to
$F_{t}\left(  x,v\right)  $ in the weak sense of measures, usually one can
pass to the limit, thanks to the non-local structure of the nonlinear terms.
In our problem, there is a main difficulty:\ $S_{t}^{N}$ is coupled with the
approximation $u_{\epsilon_{N}}^{N}$ of the Navier-Stokes component, in a
local way. The term in the Navier-Stokes equation takes the form (see system of equations \eqref{eq:PS-VNS} in Section \ref{sec:notationanmain})%
\[
\theta^{0,\epsilon_{N}}\ast\left(  u_{\epsilon_{N}}^{N}-v\right)  S_{t}^{N}%
\]
and the corresponding term in the identity satisfied by the empirical measure
$S_{t}^{N}$ (Lemma \ref{lemma:IdentityEM}), identity that should converge to the Vlasov component,
has the form%
\[
\left\langle S_{t}^{N},\left(  u_{\epsilon_{N}}^{N}-v\right)  \nabla
_{v}\varphi\right\rangle .
\]
In order to pass to the limit in the previous terms we need uniform
convergence of $u_{\epsilon_{N}}^{N}$ to $u$.

This is a demanding property which requires control of first derivatives of $u_{\epsilon_{N}%
}^{N}$. We approach it by means of the equation for the vorticity
$\omega^{N}$. This approach reveals a conceptual problem with
physical content: the presence of particles in the fluid may produce
vorticity. The estimates on the vorticity are far from being obvious, due to
the interaction with the particles. The equation for the vorticity contains
the interaction term 
\[
\frac{1}{N}\sum_{i=1}^{N}\left(  u_{\epsilon_{N}}^{N}\left(  X_{t}%
^{i,N}\right)  -V_{t}^{i,N}\right)  \nabla^{\perp}\cdot \delta_{X_{t}^{i,N}%
}^{\epsilon_{N}}
\]
where $\delta_{X_{t}^{i,N}}^{\epsilon_{N}}$ is a smoothing approximation of the delta Dirac $\delta_{X^{i,N}_{t}}$. Hence the term $\nabla^{\perp}\cdot \delta_{X_{t}^{i,N}}^{\epsilon_{N}}$ may induce a blow-up
in the estimates, a priori. This is a key conceptual difficulty we had to
overcome, among other of more technical nature.
The fact that an infinitesimal particle in a fluid may produce vorticity is the topic of recent research, see \cite{Sueur2018point}. These works are restricted to  single particle for very difficult technical reasons; it may be that some link with the present research will be possible in the future after due progresses.

Controlling $\omega^{N}$ by an energy type estimate allows us
to remove $\nabla^{\perp}$ by integration by parts, thanks to the fact that
$\nabla\omega^{N}$ has a control due to the viscous term. This
however leads to control the term%
\begin{equation}
\left\Vert \frac{1}{N}\sum_{i=1}^{N}\left(  u_{\epsilon_{N}}^{N}\left(
X_{t}^{i,N}\right)  -V_{t}^{i,N}\right)  \delta_{X_{t}^{i,N}}^{\epsilon_{N}%
}\right\Vert _{L^{2}\left(  \TT^{2}\right)  }%
.\label{term to be controlled}%
\end{equation}
This is not a simple task;\ just to mention, the trivial estimate%
\[
\leq\frac{1}{N}\sum_{i=1}^{N}\left(  u_{\epsilon_{N}}^{N}\left(  X_{t}%
^{i,N}\right)  -V_{t}^{i,N}\right)  \left\Vert \delta_{X_{t}^{i,N}}%
^{\epsilon_{N}}\right\Vert _{L^{2}\left(\TT^{2}\right)}
\]
blows-up. This introduces a new ingredient with its own difficulties, as
explained in the next subsection.

\subsubsection{The regularized empirical measure}

The way we control the term (\ref{term to be controlled}) is by introducing
the regularized empirical measure $F_{t}^{N}\left(  x,v\right)  $
\[
F^{N}_{t}(x,v) = \theta^{\eps_{N}}*S^{N}_{t}
\]
(see also Section \ref{sec:notationanmain})
somewhat in the spirit of the works of Karl Oelscheager, \cite{oelschlager1985law}. It allows us to
write
\begin{align*}
& \left\vert \frac{1}{N}\sum_{i=1}^{N}\left(  u_{\epsilon_{N}}^{N}\left(
t,X_{t}^{i,N}\right)  -V_{t}^{i,N}\right)  \delta_{X_{t}^{i,N}}^{\epsilon_{N}%
}\left(  x\right)  \right\vert \\
& \leq\left\Vert u_{\epsilon_{N}}^{N}\left(  t,\cdot\right)  \right\Vert
_{\infty}\frac{1}{N}\sum_{i=1}^{N}%
\delta_{X_{t}^{i,N}}^{\epsilon_{N}}\left(  x\right)  +\left\vert \frac{1}%
{N}\sum_{i=1}^{N}V_{t}^{i,N}\delta_{X_{t}^{i,N}}^{\epsilon_{N}}\left(
x\right)  \right\vert \\
& =\left\Vert u_{\epsilon_{N}}^{N}\left(  t,\cdot\right)  \right\Vert
_{\infty}\int_{\mathbb{R}^{2}}F_{t}%
^{N}\left(  x,v\right)  dv+\left\vert \int_{\mathbb{R}^{2}}vF_{t}^{N}\left(
x,v\right)  dv\right\vert .
\end{align*}
The proof of the last line is given in Lemma \ref{lemma:couplingrepr}.

Now the problem is to prove suitable estimates on the regularized empirical
measure $F_{t}^{N}\left(  x,v\right)  $. Opposite to $S_{t}^{N}$, whose
control essentially amounts to suitable estimates on the SDEs satisfied by
particles, a full treatment of $F_{t}^{N}\left(  x,v\right)  $ requires both
SDEs properties and PDEs arguments applied to the identity satisfied by
$F_{t}^{N}\left(  x,v\right) $ (Lemma \ref{lemma:IdentityEM}). This identity however is not closed;
commutators appear in certain computations and several technical difficulties
arise, which perhaps are new here with respect to previous literature.

\subsubsection{The cut-off and its removal}

We are able to perform the estimates outlined above only when a suitable
cut-off on velocity is introduced; see term $\chi_{R}(u)$ introduced in Section \ref{sec:preliminary} and appearing in the rest of the paper. By using the truncation in the interaction between particles and fluid we managed to produce an a priori bound independently on the number of particles $N$
\begin{equation}\label{eq:introUNR}
\norm{u^{N,R}}_{\infty}\leq C_{R}, \quad\quad \text{(Lemma \ref{lemma:tightnessmild})}
\end{equation}
which we used to obtain a suitable tightness criterion, needed for the convergence. We remark that this bound was only possible due to the presence of the cut-off, since the constant provided in \eqref{eq:introUNR} depends on the threshold $R$ of the truncation.

Therefore the preliminary result is that the PDE-particle system with cut-off converges to the PDE-PDE system with cut-off. However, by showing that the velocity field of the PDE-PDE system with the cut-off satisfies
\[
\norm{u^{R}}\leq C \quad\quad  \text{(Proposition \ref{prop:uRinfinfuniformbound})}
\]
independently on $R$, it is possible to prove that the PDE-PDE system with cut-off is also solution without cut-off.
Summarizing this first step, we can prove that the PDE-particle system with cut-off converges to the PDE-PDE system without cut-off, see Proposition \ref{prop:PSRtoVNS}.

The final problem is to prove that the cut-off can be removed also from the
approximating PDE-particle system. This seems to be a difficult question. Here we use a special trick. 

To appreciate the difficulty and the trick, think for a second to a different
problem where the approximations are not random. Assume we have proved that $u_{\epsilon_{N}}^{N}$ converges
uniformly to the limit $u$. Since $u$ is uniformly bounded by a constant
$R_{0}$ we deduce that, eventually in $N$, also $u_{\epsilon_{N}}^{N}$ is
bounded, say by $R_{0}+1$. The finite number of first terms of the sequence
$u_{\epsilon_{N}}^{N}$ are bounded as well for other reasons, hence there is a
choice of $R$ such that the cut-off can be removed, a fortori, also for the
all the approximations. This idea resemble us the method used to prove
well-posedness of 3D Navier-Stokes equations with strong rotation, see for
instance \cite{flandoli2012stochastic}.

Unfortunately this simple idea does not work when the approximations are
random. Forget about the fact that our convergence is in law;\ go to another
probability space where it is almost sure. Thus, almost surely, eventually we
may transfer the uniform bound $R_{0}$ of the limit solution to a bound
$R_{0}+1$ for the approximations. But this time the "eventually" qualification
is random! Hence we cannot bound the other finitely many terms, since they are
a random number of terms, potentially unbounded in cardinality. To make an example consider some random variables $X_{N},Y_{N},\overline{N}$ and the following set

\[
\left\{X_{N}(\omega) = Y_{N}(\omega)\quad \forall N>\overline{N}(\omega)\right\}
\]
and assume it has probability one. Can we conclude that for sufficiently large N the set $\{X_{N}(\omega)=Y_{N}(\omega)\}$ has itself probability one? Here the difficulty arise from the condition $N>\overline{N}(\omega)$ due to the randomness of $\overline{N}(\omega)$. We solve this problem by an argument which seems to be new, based also on the concept of \emph{path-by-path} uniqueness for SDEs, see \cite{flandoli2011random}, whose first major result was proved in \cite{davie2007}. 
We isolated the idea behind this reasoning into a general criterion, that we applied to transfer the convergence from the particle system where the cut-off is present, to the system without the cut-off.

The structure of this paper is the following: In Section 2 we introduce all the notation that we will use and we present our main result, Theorem \ref{teo:PStoVNS}. In Section 3 we collect some preliminary result that will be needed in the rest of the manuscript, while Section 4 is devoted to a theorem of uniqueness for the Vlasov-Navier-Stokes system. In Section 5 we prove a first intermediate result, that is the convergence of the particle system with the cut-off to the Vlasov-Navier-Stokes system without the cut off. Finally, in Section 6 we manage to remove the cut-off also from the approximating system, ending the proof of Theorem \ref{teo:PStoVNS}.
%%%%%%%%%%%%%%%%%%%%%%%%%%%%%%%%%%%%%%%%%%%%%%%%%%%%%%%%%%
\section{Notation and Main Results}\label{sec:notationanmain}
We begin this section by introducing rigorously the Vlasov-Navier-Stokes system and its associated particle model. We will always assume the framework of a filtered probability space, denoted by $(\Omega, \mathcal{F},\left\{\mathcal{F}_{t}\right\},\PP)$.
For the whole manuscript we will also work on the two dimensional torus $\TT^{2} = \RR^{2} / \mathbb{Z}^{2}$. The case of other bounded domain with boundaries is more delicate due to creation of vorticity at the boundaries. Some of the intermediate results stated here will work also in higher dimension. However, to obtain the full result, due to uniqueness and smoothness obstacles, dimension $d=2$ is needed, so we will always keep the dimension fixed for a matter of simplicity. 

We start by recalling the Vlasov-Navier-Stokes \textbf{PDE-system}
\begin{equation}\label{eq:PDE-VNS}
\begin{cases}
\partial_{t}u = \Delta u - u\cdot \nabla u - \nabla \pi - \int_{\RR^{2}} (u-v)F(x,v)\,dv &(t,x) \in [0,T]\times \TT^{2}\\
\partial_{t}F + v \cdot \nabla_{x} F + \div_{v}((u-v)F) = \frac{\sigma^{2}}{2}\Delta_{v}F  &(t,x,v) \in [0,T]\times \TT^{2}\times \RR^{2}\\
\div(u)=0,
\end{cases}
\tag{$VNS$}
\end{equation}
$\sigma > 0$, with initial condition $u(0,\cdot) = u_{0}$ and $F(0,\cdot,\cdot) = F_{0}$. 
We also introduce the \textbf{continuous-discrete Particle System} approximating \eqref{eq:PDE-VNS}:
\begin{equation}\label{eq:PS-VNS}
\begin{cases}
\partial_{t}u^{N}= \Delta u^{N}- u^{N}\cdot \nabla u^{N} -\nabla \pi^{N} -\frac{1}{N}\sum_{i=1}^{N} (u^{N}_{\eps_{N}}(X^{i,N}_{t})-V^{i,N}_{t})\delta^{\eps_{N}}_{X^{i,N}_{t}}\\
\div(u^{N})=0,\\
\begin{cases}
dX^{i,N}_{t}=V^{i,N}_{t}\,dt\\
dV^{i,N}_{t}=(u^{N}_{\eps_{N}}(X^{i,N}_{t})-V^{i,N}_{t})\,dt+\sigma dB^{i}_{t}
\end{cases}i=1,\dots,N
\end{cases}	
\tag{$PS-NS$}
\end{equation}
with initial condition 
\[
u^{N}(0,\cdot) = u_{0}, \quad  (X^{i,N}_{0},V^{i,N}_{0}) \stackrel{\mathcal{L}aw}{\sim} F(0,\cdot,\cdot)\,dx\,dv\,\, i.i.d
\] 
which is the random variables $(X^{i,N}_{0},V^{i,N}_{0})$ are independent and identically distributed with density $F(0,x,v)$.
In the previous equations, $(B^{i}_{t})_{t\geq 0}$ is a sequence of independent Brownian motions, $\theta^{0}$ is a mollifier over the torus, $\eps_{N} \in \RR^{+}$ is a sequence going to zero, and   
\[
\theta^{0,\eps_{N}}(x) := \eps_{N}^{-2}\theta^{0}\left(x/\eps_{N}\right),\quad u^{N}_{\eps_{N}} := u*\theta^{0,\eps_{N}}, \quad \delta^{\eps_{N}}_{X^{i,N}_{t}}(x) := \theta^{0,\eps_{N}}(x-X^{i,N}_{t}),
\]
All the hypothesis and requirements on the objects introduced above are collected in subsection \ref{subsec:hypo}. 

\subsection{Definition of weak solutions}
\begin{defi}[Definition of weak solution of \eqref{eq:PDE-VNS}]	
We say a pair $(u,F)$ is a weak solution of \eqref{eq:PDE-VNS} if the following conditions are satisfied:
\begin{enumerate}[a)]
\item 
\[
u \in L^{\infty}([0,T];L^{2}(\TT^{2}))\cap L^{2}([0,T];H^{1}(\TT^{2}));
\]
\[
F \in L^{\infty}([0,T];L^{1}(\TT^{2}\times \RR^{2})\cap L^{\infty}(\TT^{2}\times \RR^{2})), \quad F\geq 0;
\]
\[
F\abs{v}^{2}  \in L^{\infty}([0,T];L^{1}(\TT^{2}\times \RR^{2}));
\]
\item 
for all $\phi \in C^{\infty}([0,T]\times \TT^{2};\RR^{2})$ with $\div{\phi} = 0$ we have
%\textcolor{red}{X è la classe di funzioni test a divergenza nulla}, $u(t,x)$ satisfies the following:
\begin{multline*}
\langle u_t,\phi_t\rangle=\langle u_0,\phi_0\rangle+\int_{0}^{t} \langle u_s,\frac{\partial\phi_s}{\partial s}\rangle ds +\int_{0}^{t}\langle u_s,\Delta \phi_s\rangle ds+\int_{0}^{t} \langle u_s\cdot\nabla \phi_s,u_s\rangle ds\\
% +\int_{0}^{t}\langle\pi_s, \nabla \phi_s\rangle ds
-\int_{0}^{t}\int_{\RR^{2}}\int_{\TT^{2}}\phi_s(x)(u_s(x)-v)F_s(x,v)\,dx\,dv\,ds, 
\end{multline*}

\item for all $\psi \in C^{\infty}([0,T]\times \TT^{2}\times \RR^{2};\RR)$ with compact support in $v$ we have
\begin{eqnarray*}
\langle F_t,\psi_t\rangle=\langle F_0,\psi_0\rangle+\int_{0}^{t}\langle F_{s},\frac{\partial \psi_{s}}{\partial s} \rangle ds +\frac{\sigma^{2}}{2}\int_{0}^{t} \langle F_s,\Delta_v \psi_s\rangle ds+\nonumber\\ \int_{0}^{t}\langle F_s,v\cdot \nabla_x \psi_s\rangle ds+ 
\int_{0}^{t} \langle F_s,(u_s-v)\cdot \nabla _v \psi_s\rangle ds;
\end{eqnarray*}
\end{enumerate}
\end{defi}	

\begin{defi}[Definition of Bounded weak solution of \eqref{eq:PDE-VNS}]\label{defi:BWSPDEVNS}
We say a pair $(u,F)$ is a bounded weak solution of \eqref{eq:PDE-VNS} if it is a weak solution and 
\[
u \in L^{\infty}([0,T]\times\TT^{2}).
\]
\end{defi}	
\noindent We refer to Theorem \ref{teo:uniqueness} for a uniqueness result for bounded weak solutions.
Existence of bounded weak solutions for system \eqref{eq:PDE-VNS} will be obtained as a consequence of our main convergence result.

% Uniqueness for this class is not needed for our purpose, so we will not deepen it. However, due to smoothness of the cutoff $\chi_{R}(\cdot)$, it should be possible to treat uniqueness for bounded weak solutions of \eqref{eq:PDE-VNSR} with the same technique of Theorem \ref{teo:uniqueness}.

\subsection{The Empirical measure of the particle system}
Before stating our main result we introduce some function spaces defined as follows: given a Polish space $E$ ($E = \TT^{2}\times \RR^{2}$ in our case) we introduce
\[
\PP_{1}(E) = \left\{ \mu \text{ probability measure on } (E,\mathcal{B}(E)) \,|\, \int_{E}\abs{x}\,\mu(dx) < \infty\right\}
\]
the space of all probability measure on the Borel sets of $E$, with finite first moment. We endow this space with the Wasserstein$-1$ metric, that can be defined equivalently as
\[
\mathcal{W}_{1}(\mu,\nu) = \sup_{ [\phi]_{Lip}\leq 1}\abs{  \int_{E} \phi \, d\mu - \int_{E} \phi \,d\nu   }
\]
where $[\phi]_{Lip}$ is the usual Lipschitz seminorm. Endowed with this metric the space $\PP_{1}$ becomes a complete separable metric space, whose convergence implies the weak convergence of probability measures. 

\noindent From now on, when $\mu$ is a measure and $f$ is a function, we will denote by $\left\langle f,\mu \right\rangle$ the integration in full space of $f$ with respect to $\mu$. 

\noindent We now introduce the empirical measure of the particle system
\begin{equation}\label{eq:EM}
S^{N}_{t} = \sum_{i=1}^{N}  \delta_{(X^{i,N}_{t},V^{i,N}_{t})},%\quad S^{N,R}_{t} = \sum_{i=1}^{N}  \delta_{(X^{i,N,R}_{t},V^{i,N,R}_{t})}
\end{equation}
which are random measures on $(\Omega,\FF,\PP)$, on the space $C([0,T];\PP_{1}(\TT^{2}\times \RR^{2}))$. 
% In the following, we will describe the case of $S^{N}_{t}$, being the truncated case identical up to truncation: the reader can adapt what follows on $S^{N,R}_{t}$.
We will consider a smoothed version of the empirical measure: introduce two functions $\theta^{0}:\TT^{2}\to \RR$ and $\theta^{1}:\RR^{2}\to \RR$ which are $C^{\infty}$, non negative and integrate to one. Introduce also
\[
\theta(x,v) := \theta^{0}(x)\theta^{1}(v)
\]
 which is a function on the product space $\TT^{2}\times \RR^{2}$. Consider then 
\begin{equation}\label{eq:SEM}
\theta^{\eps_{N}}(x,v) = \eps_{N}^{-2}\theta^{0}(\eps_{N}^{-1}x)\eps_{N}^{-2}\theta^{1}(\eps_{N}^{-1}v) = \theta^{0,\eps_{N}}(x)\theta^{1,\eps_{N}}(v)
\end{equation}
 and define 
\[
F^{N}_{t}(x,v) := \theta^{\eps_{N}}*S^{N}_{t} = \frac{1}{N}\sum_{i=1}^{N} \theta^{0,\eps_{N}}(x-X^{i,N}_{t})\theta^{1,\eps_{N}}(v-V^{i,N}_{t})
\]
% and
% \[
% F^{N,R}_{t}(x,v) = \theta^{\eps_{N}}*S^{N,R}_{t}
% \]
the mollified empirical measure.
\begin{oss}	
Note that the function $\theta^{0,\eps_{N}}$ in the previous equation, appear in system \eqref{eq:PS-VNS} in the coupling term. 
\end{oss}	
In what follows and in the rest of the manuscript we will adopt the following notation for the moments on the $v$ component for the function $F$:
\[
m_{k}F(x) := \int_{\RR^{2}}\abs{v}^{k}F(x,v)\,dv,\quad M_{k}F := \int_{\TT^{2}}\int_{\RR^{2}}\abs{v}^{k}F(x,v)\,dv\,dx.
\]
where $m_{k}F(x)$ is function over $\TT^{2}$ while $M_{k}F \in \RR$.

\subsection{Main Result}\label{subsec:hypo}
We summarize all the main hypothesis of our framework:
\begin{enumerate}[1)]
\itemsep-0.3em 
\item 
 $u_{0} \in H^{2}(\TT^{2})$;
\item 
\label{hypo:F_0}
$F_{0} \in (L^{1}\cap L^{\infty})(\TT^{2}\times \RR^{2})$ and $M_{6}F_{0} < \infty$;
\item 
\label{hypo:mollifier}
$\theta(x,v) = \theta^{0}(x)\theta^{1}(v)$, $\theta^{0}$ and $\theta^{1}$ mollifiers on $\TT^{2}$ and $\RR^{2}$ respectively, such that
$\abs{\nabla{\theta^{0}(x)}} \leq \theta^{0}(x)$ and $\text{supp}(\theta^{1}) \subseteq B(0,1)$.
Also $\theta^{1}(v)$ satisfies the following symmetry assumption $
\int_{\RR^{2}}\theta^{1}(v)v = 0$;
\item
\label{hypo:scaling} 
The scaling factor $\eps_{N}$ satisfies $\eps_{N} = N^{-\beta}\,\,\text{with}\quad \beta \leq {1}/{4}$;
\end{enumerate}
\begin{oss}
We remark that hypothesis \eqref{hypo:mollifier} is needed in Lemma \ref{lemma:FNL4} to obtain the first a priori estimate on the mollified empirical measure. Regarding the scaling factor in \eqref{hypo:scaling}, this hypothesis is also needed for Lemma \ref{lemma:FNL4}: the bound on $\beta$ is strictly related on the space dimension and on the $L^{p}$ norm that is computed. In our case, we will compute the $L^{4}$ norm, and the general requirement in dimension $d$ is
\[
\beta \leq \frac{d}{3d+2}.
\]
\end{oss}
In what follows we will always use the notation $\lesssim$ to indicate that the inequality is true, up to a multiplicative constant that doesn't depend on any of the key parameters involved. To emphasize the dependency on one of those parameter we will adopt the convention $\lesssim_{X}$ to denote the dependency on the parameter $X$. Also, when needed, we will make use of the letter $C$ to mark a constant, whose value does not matter for the argument.

We are finally able to present our main result:
\begin{teo}\label{teo:PStoVNS}
Under hypothesis of subsection \ref{subsec:hypo}, the family of laws $\left\{ Q^{N}\right\}_{N\in\NN}$ of the couple $(u^{N},S^{N})_{N\in\NN}$ is tight on $C([0,T]\times\TT^{2}) \times C([0,T];\PP_{1}(\TT^{2}\times \RR^{2}))$. Moreover $\left\{ Q^{N}\right\}_{N\in\NN}$ converges weakly to $\delta_{(u,F)}$, where the couple $(u,F)$ is the unique bounded weak solution of system of equation \eqref{eq:PDE-VNS}.
\end{teo}

%%%%%%%%%%%%%%%%%%%%%%%%%%%%%%%%%%%%%%%%%%%%%%%%%%%%%%%%%%
\section{Preliminary results.}\label{sec:preliminary}
In this section we collect the basic results about our particle systems, and all the technical inequality that will be used in the rest of the paper. 

In order to obtain Theorem \ref{teo:PStoVNS}, it is necessary to introduce another coupled system of PDE-SDE, where the interaction between the particles and fluid is truncated. Introduce, for $R>0$, the cut-off function $\chi^{0}_{R}:\RR \to [0,1]$, defined as
\[
\chi^{0}_{R}(x) =
\begin{cases}
&1\quad \text{ if } x\leq R-1\\
&0\quad \text{ if } x\geq R\\
\end{cases}
\]
and that is between $0$ and $1$ for $x \in (R-1,R)$ and that is $C^{\infty}$ in the entire line. Define also $\chi_{R}(u) = \chi_{R}^{0}(\norm{u}_{L^{\infty}(\TT^{2})})$.
With this choice of notation we have
\[
\norm{u \chi_{R}(u)}_{L^{\infty}(\TT^{2})} \leq R.
\]
Introduce now the \textbf{truncated PDE-system}:
\begin{equation}\label{eq:PDE-VNSR}
\begin{cases}
\partial_{t}u^{R} = \Delta u^{R} - u^{R}\cdot \nabla u^{R} - \nabla \pi - \int_{\RR^{2}} (u^{R}-v)\chi_{R}(u^{R}_{t})F^{R}(x,v)\,dv\\
\partial_{t}F^{R} = \frac{\sigma^{2}}{2}\Delta_{v}F^{R} - v \cdot \nabla_{x} F^{R} - \div_{v}((u^{R}\chi_{R}(u^{R})-v)F^{R})\\
\div(u^{R})=0,
\end{cases}
\tag{$VNS^{R}$}
\end{equation}
with the same initial conditions as system \eqref{eq:PDE-VNS}. We also introduce the \textbf{continuous-discrete truncated Particle System} approximating \eqref{eq:PDE-VNSR}:
\begin{equation}\label{eq:PS-VNSR}
\begin{cases}
\partial_{t}u^{N,R}= \Delta u^{N,R}- u^{N,R}\cdot \nabla u^{N,R} -\nabla \pi^{N,R} \\
\quad\quad\quad\quad\quad\quad\quad
 -\frac{1}{N}\sum_{i=1}^{N} (u^{N,R}_{\eps_{N}}(X^{i,N,R}_{t})-V^{i,N,R}_{t})\chi_{R}(u^{N,R}_{t})\,\delta^{\eps_{N}}_{X^{i,N,R}_{t}}\\
\div(u^{N,R})=0,\\
\begin{cases}
dX^{i,N,R}_{t}=V^{i,N,R}_{t}\,dt\\
dV^{i,N,R}_{t}=(u^{N,R}_{\eps_{N}}(X^{i,N,R}_{t})\chi_{R}(u^{N,R}_{t})-V^{i,N;R}_{t})\,dt+\sigma dB^{i}_{t}
\end{cases}i=1,\dots,N
\end{cases}	
\tag{$PS^{R}-NS^{R}$}
\end{equation}
using the same notation and initial condition as \eqref{eq:PS-VNS}.

\begin{defi}[Definition of bounded weak solution of \eqref{eq:PDE-VNSR}]\label{defi:WSPDEVNSR} 
We say a pair $(u^{R},F^{R})$ is a bounded weak solution of \eqref{eq:PDE-VNSR} if the following condition are satisfied:
\begin{enumerate}[a)]
\item 
\[
u^{R} \in L^{\infty}([0,T]\times\TT^{2})\cap L^{2}([0,T];H^{1}(\TT^{2}));
\]
\[
F^{R} \in L^{\infty}([0,T];L^{1}(\TT^{2}\times \RR^{2})\cap  L^{\infty}(\TT^{2}\times \RR^{2})), \quad F^{R}\geq 0;
\]
\[
F\abs{v}^{2} \in L^{\infty}([0,T];L^{1}(\TT^{2}\times \RR^{2}));
\]
\item 
for each divergence free, $C^{\infty}$ vector field $\phi:[0,T]\times\TT^{2}\to\RR^{2}$ we have

%\textcolor{red}{X è la classe di funzioni test a divergenza nulla}, $u(t,x)$ satisfies the following:
\begin{eqnarray*}
\langle u^{R}_t,\phi_t\rangle=\langle u^{R}_0,\phi_0\rangle+\int_{0}^{t} \langle u^{R}_s,\frac{\partial\phi_s}{\partial s}\rangle ds +\int_{0}^{t}\langle u^{R}_s,\Delta \phi_s\rangle ds+\int_{0}^{t} \langle u^{R}_s\cdot\nabla \phi_s,u^{R}_s\rangle ds\\
+\int_{0}^{t}\langle\pi_s, \nabla \phi_s\rangle ds
-\int_{0}^{t}\int_{\RR^{2}}\int_{\TT^{2}}\phi_s(x)(u^{R}_s(x)-v)\chi_{R}(u^{R}_{s})F^{R}_s(x,v)\,dx\,dv\,ds, 
\end{eqnarray*}

\item for each $C^{\infty}$ function $\psi:[0,T]\times \TT^2\times \RR^{2}\to\RR$, we have
\begin{eqnarray*}
\langle F^{R}_t,\psi_t\rangle=\langle F^{R}_0,\psi_0\rangle+\int_{0}^{t}\langle F^{R}_{s},\frac{\partial \psi_{s}}{\partial s} \rangle ds +\frac{\sigma^{2}}{2}\int_{0}^{t} \langle F^{R}_s,\Delta_v \psi_s\rangle ds+\nonumber\\ \int_{0}^{t}\langle F^{R}_s,v\cdot \nabla_x \psi_s\rangle ds+ 
\int_{0}^{t} \langle F^{R}_s,(u^{R}_s\chi_{R}(u^{R}_{s})-v)\cdot \nabla _v \psi_s\rangle ds;
\end{eqnarray*}
\end{enumerate}
\end{defi}
Applying the maximum principle to system of equation \eqref{eq:PDE-VNSR} we have
\[
\norm{F^{R}(t,x,v)}_{L^{p}(\TT^{2}\times \RR^{2})} \leq C_{T} \norm{F_{0}(x,v)}_{L^{p}(\TT^{2}\times \RR^{2})}\quad \forall p >1
\]
so that 
\[
\norm{F^{R}(t,x,v)}_{L^{\infty}(\TT^{2}\times \RR^{2})}\leq C
\]
independently on $R$.
We now introduce the empirical measure of the truncated particle system 
\begin{equation*}
S^{N,R}_{t} = \sum_{i=1}^{N}  \delta_{(X^{i,N,R}_{t},V^{i,N,R}_{t})}
\end{equation*}
and its associated mollified empirical measure
\[
F^{N,R}_{t}(x,v) = \theta^{\eps_{N}}*S^{N,R}_{t}.
\]

\noindent We now recall the identity satisfied by the empirical measure $S^{N}_{t}$. %(resp. by  $S^{N,R}_{t}$)
\begin{lem}\label{lemma:IdentityEM}
For every test function $\phi : \TT^{2}\times \RR^{2} \to \RR$ the empirical measure $S^{N}_{t}$ % (resp. $S^{N,R}_{t}$) 
satisfy the following identity
\begin{multline*}
d\langle S^{N}_{t},\phi \rangle = \langle  S^{N}_{t},v\cdot \nabla_{x}\phi\rangle\,dt +  \langle S^{N}_{t}, (u^{N}_{\eps_{N}}-v)\cdot \nabla_{v}\phi  \rangle\,dt \\
+\frac{\sigma^{2}}{2}\langle S^{N}_{t}, \Delta_{v}\phi\rangle\,dt
+dM^{N,\phi}_{t},
\end{multline*}
with 
\[
M^{N,\phi}_{t}=\frac{\sigma}{N}\sum_{i=1}^{N}\int_{0}^{t}\nabla_{v}\phi\left(X^{i,N}_{t},V^{i,N}_{t}\right)\cdot dB^{i}_{t}.
\]
Moreover $F^{N}_{t}(x,v)=(\theta^{\eps_{N}}*S^{N}_{t})(x,v)$ satisfies:
\begin{multline*}
dF^{N}_{t}=\frac{\sigma^{2}}{2}\Delta_{v}F^{N}_{t}-\div_{v}(\theta^{\eps_{N}}*(u^{N}_{\eps_{N}}-v)S^{N}_{t})\,dt\\ 
- \div_{x}(\theta^{\eps_{N}}*vS^{N}_{t})\,dt+dM_{t}^{N,\eps_{N}},
\end{multline*}
with $M^{N,\eps_{N}}_{t}=M^{N,\theta^{\eps_{N}}\left(x-\cdot,v-\cdot\right)}_{t}$.
\end{lem}
\begin{proof}
The first part follows easily by applying It\^o formula to $\phi(X^{i,N}_{t},V^{i,N}_{t})$ and using linearity. The second part follows by taking $\phi(\cdot,\cdot) = \theta^{\eps_{N}}(x-\cdot,v-\cdot)$.
\end{proof}
\noindent The analogue of the previous result holds for the empirical measure of the truncated system $S^{N,R}$, as well as for it mollified version $F^{N,R}$.
We now state the kinetic energy balance for the truncated system:
\begin{lem}\label{lemma:energyPS-VNSR}
With the previous notation, setting
\[
\mathcal{E}^{N}(t) = \frac{1}{2}\int_{\TT^{2}}\abs{u^{N,R}_{t}(x)}^{2}\,dx + \frac{1}{2N}\sum_{i=1}^{N} \abs{V^{i,N,R}_{t}}^{2}
\]
one has formally
\begin{multline*}
\frac{1}{2}\frac{d}{dt}\mathcal{E}^{N}(t) +\int_{\TT^{2}}\abs{\nabla u^{N,R}_{t}(x)}\,dx\,dt +\\+ \frac{1}{N}\sum_{i=1}^{N}\abs{u^{N,R}_{\eps_{N}}(X^{i,N,R}_{t})\chi_{R}(u^{N,R}_{t})-V^{i,N,R}_{t}}^{2}\,dt\leq \\
\leq \frac{2\sigma^{2}}{2}\,dt+ \frac{\sigma}{N}\sum_{i=1}^{N}V^{i,N,R}_{t}  \cdot dB^{i}_{t}.
\end{multline*}
\end{lem}
\begin{proof}
The lemma follows by It\^o formula and by classical energy estimates for $u^{N,R}$.
\end{proof}
\begin{oss}
The last inequality guarantees that, even if the truncated system has no direct interpretation for the dynamics of particle-fluid, it maintains the basics physical properties such as the conservation of the kinetic energy in the average. 
\end{oss}
An analogue of the previous result holds for the limit PDE system \eqref{eq:PDE-VNSR}, as well as for \eqref{eq:PDE-VNS}. We state it in the case of system \eqref{eq:PDE-VNSR} and omit the proof, which is classical.
\begin{lem}\label{lemma:energyPDE-VNSR}
If $(u^{R},F^{R})$ is a weak solution of \eqref{eq:PDE-VNSR} then the following holds:
setting
\[
\mathcal{E}(t) = \frac{1}{2}\left(\int_{\TT^{2}}\abs{u^{R}_{t}}^{2}\,dx + \int_{\RR^{2}}\int_{\TT^{2}}\abs{v}^{2}F^{R}_{t}\,dx\,dv\right),
\]
one has
\[
\frac{d}{dt}\mathcal{E}(t) + \int_{\TT^{2}}\hspace{-0.1cm}\abs{\nabla u^{R}_{t}}^{2}\,dx + \int_{\RR^{2}}\int_{\TT^{2}}F^{R}_{t}\abs{u^{R}_{t}-v}^{2}\chi_{R}(u^{R}_{t})\,dx\,dv = \frac{\sigma^{2}}{2}\norm{F_{0}}_{L^{1}(\RR^{2}\times \TT^{2})}\hspace{-0.1cm}.
\]
Moreover there exists a constant $C$, independent on $R$ such that 
\[
\int_{0}^{T}\int_{\RR^{2}}\int_{\TT^{2}}\abs{v}^{2}F^{R}_{t}\,dx\,dv\,dt \leq C.
\]
\end{lem}
\begin{oss}\label{oss:uL2Lp}
By the previous lemma we have a bound on $u^{R}$ in the norm $L^{2}([0,T];H^{1}(\TT^{2}))$ independently on $R$. By Sobolev embedding in dimension two we also have an uniform bound with respect to $R$ on $u^{R}$ in the space $L^{2}([0,T];L^{p}(\TT^{2}))$ for all $p < \infty$. 
\end{oss}
% \begin{oss}
% Applying the maximum principle to the equation for $F$ in \eqref{eq:PDE-VNS} (analogously for \eqref{eq:PDE-VNSR}) we have
% \[
% \norm{F_{t}}_{L^{p}(\TT^{2}\times \RR^{2})} \leq C(T) \norm{F_{0}}_{L^{p}(\TT^{2}\times \RR^{2})},\quad \forall p > 1.
% \]
% \end{oss}

We now collect all the inequalities concerning the marginal distributions of the function $F$: some of them are classical, see \cite{Hamdache1998}, \cite{CYu}, while others have been used in \cite{oldVNS}. 
\begin{lem}\label{lemma:marginalinequ} 
If $F$ is positive, defined on $\TT^{2}\times \RR^{2}$, the followings hold
\begin{enumerate}
\itemsep-1em 
\item 
\[
\norm{m_{0}F}_{L^{2}(\TT^{2})}^{2}\lesssim (\norm{F}_{L^{\infty}(\TT^{2}\times \RR^{2})}+1)^{2}M_{2}F,
\]
\[
\norm{m_{0}F}_{L^{4}(\TT^{2})}^{4}\lesssim (\norm{F}_{L^{\infty}(\TT^{2}\times \RR^{2})}+1)^{4}M_{6}F;
\]
\item 
\[
\norm{m_{1}F}_{L^{2}(\TT^{2})}^{2}\lesssim (\norm{F}_{L^{\infty}(\TT^{2}\times \RR^{2})}+1)^{2}M_{4}F;
\]
\item 
\[
\norm{m_{0}F}_{L^{2}(\TT^{2})}^{2}\lesssim \norm{F}_{L^{4}(\TT^{2}\times \RR^{2})}^{4}+M_{3}F;
\]
\item 
\[
\norm{m_{1}F}_{L^{2}(\TT^{2})}^{2}\lesssim \norm{F}_{L^{4}(\TT^{2}\times \RR^{2})}^{4}+M_{6}F;
\]

\item For all $k < k'$
\[
M_{k}F \lesssim \norm{F}_{L^{1}(\TT^{2}\times \RR^{2})} + M_{k'}F.
\]
\end{enumerate}
\end{lem}
\begin{proof}
All the inequalities follow the same strategy: 1. and 2. are classical, see \cite{Hamdache1998}, while the proof of 3. can be found in \cite{oldVNS}, so we only outline the main idea.
For inequality 1. and 3. consider the following decomposition
\begin{multline*}
\int_{\RR^{2}} F\,dv = \int_{\abs{v}\leq r(x)} F\,dv + \int_{\abs{v}>r(x)}F\,dv \\\leq \int_{\abs{v}\leq r(x)} F\,dv + \frac{1}{r(x)^{k}}\int_{\abs{v}>r(x)}\abs{v}^{k}F\,dv
\end{multline*}
where $r(x)$ will be chosen later. Now one estimates the integral on the ball of radius $r(x)$ using the infinity norm of $F$ for inequality 1. or using Holder inequality to obtain $\norm{F}_{L^{4}}$ for inequality 3. Taking the square both sides and integrating on $\TT^{2}$ leads to the desired result by choosing $r(x)$ in the proper manner to group all the terms. For inequality 2. and 4. one has just to decompose $\int \abs{v} F \,dv$ and apply the same strategy, while for $5$ is enough to take $r(x) \equiv 1$.
\end{proof}

\begin{oss}
Inequality 3. and 4. will be used to prove a first tightness result in Section \ref{sec:limittruncated}. They are a variant of 1. and 2. avoiding the use of the infinity norm, which is not available when dealing with the mollified empirical measures of the particle system. Inequalities 1. and 2. will be used in Section \ref{sec:limittruncated} in order to prove a bound on the infinity norm of $u^{R}$, while 5. will be used in the next lemma.
\end{oss}
We now state and prove a variant of Lemma 2.1 in \cite{Hamdache1998}. This variation is needed due to the presence of the noise on the diffusion on the particle velocity, i.e. the presence of $\Delta_{v}$ in the equation for $F^{R}$. 
% For convenience we state the lemma in the case of system \eqref{eq:PDE-VNS}, but it remains true also for \eqref{eq:PDE-VNSR}.
\begin{lem}\label{lemma:momentbound}
If $(u^{R},F^{R})$ is a bounded weak solution of \eqref{eq:PDE-VNSR}, 
 $k>2$ and if $M_{k}F_{0}$ is finite, then there exists a constant $C_{k}$, independent on $R$, such that
\[
\sup_{t\in[0,T]} M_{k}F^{R}_{t} \leq C_{k}.
\] 
The same result holds for any $(u,F)$ weak solutions of \eqref{eq:PDE-VNS}.
\end{lem}
\begin{proof}
In this proof we will omit the superscript $R$ in $(u^{R},F^{R})$ to short the notation. We start by computing 
\begin{multline*}
\frac{d}{dt}\int_{\RR^{2}}\int_{\TT^{2}}\abs{v}^{k}F_{t}\,dx\,dv \lesssim \int_{\TT^{2}}\abs{u(t,x)}\int_{\RR^{2}}\abs{v}^{k-1}F_{t}\,dv\,dx + \int_{\RR^{2}}\int_{\TT^{2}}\abs{v}^{k}F_{t}\,dx\,dv
 \\+ \int_{\RR^{2}}\int_{\TT^{2}}\abs{v}^{k-2}F_{t}\,dx\,dv.
\end{multline*}
Following \cite{Hamdache1998} we have
\[
\int_{\TT^{2}}\abs{u(t,x)}\int_{\RR^{2}}\abs{v}^{k-1}F_{t}\,dv\,dx \lesssim \norm{u_{t}}_{L^{k+2}(\TT^{2})}\left( \int_{\RR^{2}}\int_{\TT^{2}}\abs{v}^{k}F_{t}\,dx\,dv \right)^{1-\frac{1}{k+2}},
\]
while, using Lemma \ref{lemma:marginalinequ} inequality 5. we have
\[
 \int_{\RR^{2}}\int_{\TT^{2}}\abs{v}^{k-2}F_{t}\,dx\,dv \leq   \int_{\RR^{2}}\int_{\TT^{2}}\abs{v}^{k}F_{t}\,dx\,dv + \norm{F_{t}}_{L^{1}(\TT^{2}\times \RR^{2})}.
\]
Hence we get
\begin{multline*}
M_{k}F_{t} \lesssim M_{k}F_{0} + \int_{0}^{t} \norm{u_{s}}_{L^{k+2}(\TT^{2})}\left( \int_{\RR^{2}}\int_{\TT^{2}}\abs{v}^{k}F_{s}\,dx\,dv \right)^{1-\frac{1}{k+2}}\,dt \\
+ \int_{0}^{t}M_{k}F_{s}\,ds +C.
\end{multline*}
We now note that
\[
\left( \int_{\RR^{2}}\int_{\TT^{2}}\abs{v}^{k}F_{s}\,dx\,dv \right)^{1-\frac{1}{k+2}} \leq C \left( \int_{\RR^{2}}\int_{\TT^{2}}\abs{v}^{k}F_{s}\,dx\,dv +1 \right)	,
\]
hence we obtain
\[
M_{k}F_{t} \leq M_{k}F_{0} + C\int_{0}^{t} (\norm{u_{s}}_{L^{k+2}(\TT^{2})}+1) M_{k}F_{s}\,ds +C\int_{0}^{t} (\norm{u_{s}}_{L^{k+2}(\TT^{2})}+1)\,ds\leq
\]
\[
\leq C(M_{k}F_{0} + \norm{u}_{L^{2}([0,T];L^{k+2}(\TT^{2}))})+ C\int_{0}^{t} (\norm{u_{s}}_{L^{k+2}(\TT^{2})}+1) M_{k}F_{s}\,ds.
\]
We conclude by classical Gronwall Lemma applied to the function $M_{k}F_{t}$ and by Remark \ref{oss:uL2Lp}.
\end{proof}
\subsection{Maximum principle for weak solutions of the linear Vlasov-Fokker-Plank equation}\label{subsec:maximumprinciple}
We now focus on boundedness of weak solutions for the linear Vlasov-Fokker-Plank equation
\[
\partial_{t}F + v \cdot \nabla_{x}F + \div_{v}(a(t,x,v)F) = \Delta_{v}F.
\] 
Boundedness of solutions will be fundamental in the latter, when we will prove that the limit points, in the appropriate sense, of  particle system \eqref{eq:PS-VNSR} are supported on bounded weak solutions of \eqref{eq:PDE-VNS}. 
While this topic is classical in the case of smooth solutions, the case of weak solutions is more delicate. 
What follows is mainly an adaptation of the work \cite{degondMax}, Appendix A, Proposition A.3. \\
In that work the author assumed the vector field $a$ to be 
\[
a \in L^{\infty}([0,T]\times \TT^{2}\times \RR^{2}),\quad \div_{v}(a)\in L^{\infty}([0,T]\times \TT^{2}\times \RR^{2}),
\]
considered solutions $F$ belonging to the set
\[
Y\hspace{-0.1cm}:=\hspace{-0.1cm}\bigg\{ F \in L^{2}([0,T]\times \TT^{2};H^{1}(\RR^{2})) \text{ s.t. } \partial_{t}F+v\cdot \nabla_{x}F \in L^{2}([0,T]\times \TT^{2};H^{-1}(\RR^{2}))  \bigg\}
\]
and proved that
\[
\norm{F_{t}}_{L^{\infty}(\TT^{2}\times \RR^{2})}\leq C \norm{F_{0}}_{L^{\infty}(\TT^{2}\times \RR^{2})},
\]
namely, a maximum principle. 
In our case, we have to consider 
\begin{equation}\label{eq:a=u-v}
a(t,x,v) =  u(t,x) - v
\end{equation}
hence, we cannot apply directly the result presented in \cite{degondMax} since the function $a(t,x,v)$ is not globally bounded. However, it is possible to recover the same result by considering some estimates on higher moments for the function $F$. If $a$ is as in \eqref{eq:a=u-v}, where $u$ is uniformly bounded, we consider
\begin{multline*}
\widetilde{Y}:= \bigg\{ F \in L^{2}([0,T]\times \TT^{2};H^{1}(\RR^{2})) \text{ s.t.   } vF \in L^{2}([0,T]\times \TT^{2}\times \RR^{2}),\\ \partial_{t}F+v\cdot \nabla_{x}F \in L^{2}([0,T]\times \TT^{2};H^{-1}(\RR^{2}))  \bigg\}.
\end{multline*}
Namely, in this setup the same result proved in \cite{degondMax} still works, provided that one considers solutions satisfying 
\begin{equation*}
\int_{0}^{T} \int_{\TT^{2}}\int_{\RR^{2}}\abs{v}^{2}F_{s}^{2}\,dx\,dv\,ds<\infty.
\end{equation*}
Without going into the details of this adaptation, we only remark that this additional condition is achievable under our hypothesis, since
\begin{multline*}
\int_{0}^{T} \int_{\TT^{2}}\int_{\RR^{2}}\abs{v}^{2}F_{s}^{2}\,dx\,dv\,ds= \int_{0}^{T} \int_{\TT^{2}}\int_{\RR^{2}} \abs{v}^{2}F^{\frac{1}{2}}_{s}F^{\frac{3}{2}}_{s}\,dx\,dv\,ds \\ 
\leq \left(\int_{0}^{T} \int_{\TT^{2}}\int_{\RR^{2}} \abs{v}^{4}F_{s}\,dx\,dv\,ds \right)^{\frac{1}{2}}\left(\int_{0}^{T} \int_{\TT^{2}}\int_{\RR^{2}}F_{s}^{3} \,dx\,dv\,ds \right)^{\frac{1}{2}},
\end{multline*}
and we will show how to control the last two terms when needed.

%%%%%%%%%%%%%%%%%%%%%%%%%%%%%%%%%%%%%%%%%%%%%%%%%%%%%%%%%%
\section{Uniqueness for bounded weak solutiosn of system of equations \eqref{eq:PDE-VNS}}
In this section we isolate a first major result needed to prove Theorem \ref{teo:PStoVNS}. We preferred to isolate it here, given that it can have some interest by itself. 
We present an uniqueness result for \eqref{eq:PDE-VNS} in the class of bounded weak solutions (Definition \eqref{defi:BWSPDEVNS}). This result will be needed in what follows in order to prove that converging subsequence of laws $(u^{N_{k}},S^{N_{k}})$ are all supported on the same limit, which are in fact weak solutions of \eqref{eq:PDE-VNS}. 

Before going into the details of this Theorem let us make some remark about the hypothesis. We first highlight that the boundedness of solutions on the fluid component is strictly needed: we will make frequent of the fact that $u \in L^{\infty}([0,T]\times \TT^{2})$  in order to close some of the estimates needed to end the proof. We also remark that, even if in the proof we used the uniform bound $\norm{u}_{L^{\infty}([0,T]\times \TT^{2})}$, with a bit more effort it is possible to complete the proof using only $u \in L^{2}([0,T];L^{\infty}(\TT^{2}))$.
However, given that in our framework we proved existence of solutions uniformly bounded in time and space, we followed this approach instead. 
Moreover we also highlight that, from the point of view of integrability of weak derivatives, we manage to restrict our hypothesis as little as possible: namely we require only 
\[
u\in L^{\infty}([0,T]\times \TT^{2})\cap L^{2}([0,T];H^{1}(\TT^{2}))
\]
without any assumption on the second derivative of $u$.  Also in the following proof we will make frequent use of Gagliardo-Niremberg inquality in dimension two
\[
\norm{u}_{L^{p}}\lesssim \norm{u}_{L^{2}}^{\frac{2}{p}}\norm{\nabla u}_{L^{2}}^{\frac{2}{q}}
\]
where $\frac{1}{p}+\frac{1}{q}=\frac{1}{2}$. However, this is only needed to minimize the hypothesis on $M_{k}F_{0}$, required to be finite only for some $k$ strictly bigger than $4$. One could have used the classical Ladyzhenskaya's  inequality ($p=q=4$) with the downside of having to require higher-order moments to be finite. 
\noindent The proof of this result is mainly inspired by the work \cite{3luo}.

\begin{teo}\label{teo:uniqueness}
Let $(u_{1},F_{1})$ and $(u_{2},F_{2})$ be two bounded weak solutions (Definition \eqref{defi:BWSPDEVNS}) with the same initial conditions, of system  \eqref{eq:PDE-VNS}. If
\[
M_{4+\eps}F_{i}(0) < \infty
\]
for some $\eps > 0$, then $u_{1} = u_{2}$ and $F_{1}=F_{2}$.
\end{teo}
\begin{proof}
We introduce the new variables $F = F_{1}-F_{2}$ and $u = u_{1}-u_{2}$. Then the pair $(u,F)$ satisfies in the weak sense
\[
\partial_{t} u = \Delta u - u\cdot \nabla u_{1} - u_{2}\cdot \nabla u- \nabla(\pi_{1}-\pi_{2})-\int_{\RR^{2}}(u F_{1}+u_{2}F-vF)\,dv, 
\]
\[
\partial_{t}F = \Delta_{v}F -v\cdot \nabla_{x}F -\div_{v}(u F_{1}+u_{2}F-vF)
\]
with $(u(0,\cdot),F(0,\cdot,\cdot)) = 0$. 
We will prove uniqueness by applying Gronwall Lemma to the quantity
\[
\norm{u_{t}}_{L^{2}(\TT^{2})}^{2}+\norm{\ang{v}^{k}F_{t} }_{L^{2}(\TT^{2}\times \RR^{2})}^{2}
\]
for some $k>2$ which will be chosen later and where $\ang{v} = (1+\abs{v}^{2})^{\frac{1}{2}}$.

We start by estimating $\norm{u_{t}}_{L^{2}(\TT^{2})}^{2}$: computing the time derivative we have
\begin{multline}\label{eq:uniqueness u^2}
\frac{d}{dt}\norm{u}_{L^{2}(\TT^{2})}^{2} + \norm{\nabla u}_{L^{2}(\TT^{2})}^{2} \lesssim \\ 
- \int_{\TT^{2}} u (u \cdot \nabla u_{1})\,dx -\int_{\TT^{2}} u (u_{2} \cdot \nabla u)\,dx \\
-\int_{\TT^{2}} u \int_{\RR^{2}} v F\,dvdx - \int_{\TT^{2}} u \int_{\RR^{2}} u F_{1}\,dvdx -\int_{\TT^{2}} u \int_{\RR^{2}} u_{2} F\,dvdx.
\end{multline}
Integrating by parts the term 
\[
\int_{\TT^{2}} u (u_{2} \cdot \nabla u)\,dx = 0
\]  
vanishes, while the term 
\[
- \int_{\TT^{2}} u \int_{\RR^{2}} u F_{1}\,dvdx = - \int_{\TT^{2}}  \int_{\RR^{2}} u^{2} F_{1}\,dvdx  \leq 0
\] 
can be neglected due to positivity of $F_{1}$. Hence we can estimate the remaining terms as 
\begin{multline*}
\eqref{eq:uniqueness u^2} \lesssim - \int_{\TT^{2}} u (u \cdot \nabla u_{1})\,dx  -\int_{\TT^{2}} u \int_{\RR^{2}} v F\,dvdx \\ -\int_{\TT^{2}} u \int_{\RR^{2}} u_{2} F\,dvdx = (I)+(II)+(III).
\end{multline*}
% \[
%   - \int_{\TT^{2}} u (u \cdot \nabla u_{1})\,dx  -\int_{\TT^{2}} u \int_{\RR^{2}} v F\,dvdx -\int_{\TT^{2}} u \int_{\RR^{2}} u_{2} F\,dvdx = (I)+(II)+(III)
% \]
where
\begin{multline*}
(I) \leq \int_{\TT^{2}}\abs{u}\abs{\nabla u} \abs{u_{1}}\,dx \leq \norm{u_{1}}_{\infty}\int_{\TT^{2}}\abs{u}\abs{\nabla u}\,dx \\\lesssim \frac{1}{\delta}\norm{u}_{L^{2}(\TT^{2})}^{2}+\delta\norm{\nabla u}_{L^{2}(\TT^{2})}^{2}
\end{multline*}
and $\delta >0 $ can be taken arbitrarily small. 
\begin{multline*}
(II) \leq \int_{\RR^{2}}\int_{\TT^{2}}\abs{u}\abs{v}F\,dxdv\leq \int_{\RR^{2}}\int_{\TT^{2}}\frac{\abs{u}}{\ang{v}^{k-1}}\ang{v}^{k}F\,dxdv\\
 \leq \int_{\TT^{2}}\abs{u}^{2}dx \int_{\RR^{2}}\frac{1}{\ang{v}^{2(k-1)}}\,dv  + \norm{\ang{v}^{k}F}_{L^{2}(\TT^{2}\times \RR^{2})}^{2} \\ \lesssim  \norm{u}_{L^{2}(\TT^{2})}^{2} +\norm{\ang{v}^{k}F}_{L^{2}(\TT^{2}\times \RR^{2})}^{2},
\end{multline*}
because $2(k-1)>2$ being $k>2$.
\begin{multline*}
(III) \leq \int_{\RR^{2}}\int_{\TT^{2}} \abs{u}\abs{u_{2}}F\,dxdv \leq \norm{u_{2}}_{\infty}\int_{\RR^{2}}\int_{\TT^{2}} \frac{\abs{u}}{\ang{v}^{k}}\ang{v}^{k}F\,dxdv\\
\lesssim \hspace{-0.1cm}\int_{\TT^{2}}\hspace{-0.1cm}\abs{u}^{2}dx\hspace{-0.1cm}\int_{\RR^{2}}\frac{1}{\ang{v}^{2k}}dv + \norm{\ang{v}^{k}F}_{L^{2}(\TT^{2}\times \RR^{2})}^{2}\hspace{-0.1cm}\lesssim \norm{u}_{L^{2}(\TT^{2})}^{2} + \norm{\ang{v}^{k}F}_{L^{2}(\TT^{2}\times \RR^{2})}^{2}\hspace{-0.1cm}.
\end{multline*}
This ends the estimate for $\norm{u}_{L^{2}(\TT^{2})}^{2}$. Concerning $\norm{\ang{v}^{k}F}_{L^{2}(\TT^{2}\times \RR^{2})}^{2}$ we proceed by computing the time derivative
\begin{multline}\label{eq:uniqueness <v>^kF}
\frac{d}{dt}\norm{\ang{v}^{k}F}_{L^{2}(\TT^{2}\times \RR^{2})}^{2} + \norm{\ang{v}^{k}\nabla_{v}F}_{L^{2}(\TT^{2})}^{2}\lesssim \\
+\int_{\RR^{2}}\int_{\TT^{2}}\ang{v}^{2k-2}F^{2}\,dxdv - \int_{\RR^{2}}\int_{\TT^{2}}\ang{v}^{2k}F v\cdot \nabla_{x}F\,dxdv\\
-\int_{\RR^{2}}\int_{\TT^{2}}\ang{v}^{2k}F\div_{v}(u_{2}F)\,dxdv-\int_{\RR^{2}}\int_{\TT^{2}}\ang{v}^{2k}F\div_{v}(uF_{1})\,dxdv\\ +\int_{\RR^{2}}\int_{\TT^{2}}\ang{v}^{2k}F\div_{v}(vF)\,dxdv.
\end{multline}
The first term on the r.h.s. can be estimated with $\norm{\ang{v}^{k}F}_{L^{2}(\TT^{2}\times \RR^{2})}^{2}$, being $\ang{v}\geq 1$, while the second can be seen to be zero with a standard integration by parts argument.
Hence, what is left from \eqref{eq:uniqueness <v>^kF} is
\begin{multline*}
-\int_{\RR^{2}}\int_{\TT^{2}}\ang{v}^{2k}F\div_{v}(u_{2}F)\,dxdv-\int_{\RR^{2}}\int_{\TT^{2}}\ang{v}^{2k}F\div_{v}(uF_{1})\,dxdv\\ +\int_{\RR^{2}}\int_{\TT^{2}}\ang{v}^{2k}F\div_{v}(vF)\,dxdv = (IV) + (V) + (VI).
\end{multline*}
Now we proceed by treating each term separately:
\begin{multline*}
(IV) = -\frac{1}{2}\int_{\RR^{2}}\int_{\TT^{2}} \ang{v}^{2k} u_{2}\cdot \nabla_{v}F^{2}\,dxdv \leq \int_{\RR^{2}}\int_{\TT^{2}} \ang{v}^{2k-1}\abs{u_{2}}F^{2}\,dxdv \\\lesssim 
\norm{u_{2}}_{\infty}\int_{\RR^{2}}\int_{\TT^{2}} \ang{v}^{2k} F^{2}\,dxdv\lesssim   \norm{\ang{v}^{k}F}_{L^{2}(\TT^{2}\times \RR^{2})}^{2}.
\end{multline*}
\begin{multline*}
(V) = \int_{\RR^{2}}\int_{\TT^{2}} \nabla_{v}(\ang{v}^{2k}F) \cdot u F_{1}\,dxdv \leq \int_{\RR^{2}}\int_{\TT^{2}} \ang{v}^{2k-1}F\, \abs{u}\, F_{1}\,dxdv \\+ \int_{\RR^{2}}\int_{\TT^{2}}\ang{v}^{2k} \abs{\nabla_{v}F}\,\abs{u}\, F_{1}\,dxdv.
\end{multline*}
The first term on the r.h.s. of the last inequality can be treated in the following way
\begin{multline}\label{eq:uniqueness V_1}
\int_{\RR^{2}}\int_{\TT^{2}} \ang{v}^{2k-1}F\, \abs{u}\, F_{1}\,dxdv = \int_{\RR^{2}}\int_{\TT^{2}} \left(\ang{v}^{k}F \right)\frac{\abs{u}}{\ang{v}}\left(\ang{v}^{k}F_{2}\right)\,dxdv \\
\leq \norm{\ang{v}^{k}F}_{L^{2}(\TT^{2}\times \RR^{2})} \norm{u}_{L^{p}(\TT^{2})} \left(\int_{\RR^{2}}\frac{1}{\ang{v}^{p}}\,dv\right)^{\frac{1}{p}}\norm{\ang{v}^{k}F_{2}}_{L^{q}(\TT^{2}\times \RR^{2})}
\end{multline}
where $p$ and $q$ are such that $\frac{1}{p}+\frac{1}{q}+\frac{1}{2}=1$. Note that $p>2$ so that $1/\ang{v}^{p}$ is integrable in dimension two. Applying Gagliardo-Niremberg inquality to the previous identity we have 
\begin{multline*}
\eqref{eq:uniqueness V_1} \leq \norm{\ang{v}^{k}F}_{L^{2}(\TT^{2}\times \RR^{2})} \norm{u}_{L^{2}(\TT^{2})}^{\frac{2}{p}}\norm{\nabla u}_{L^{2}(\TT^{2})}^{\frac{2}{q}} \norm{\ang{v}^{k}F_{2}}_{L^{q}(\TT^{2}\times \RR^{2})}\\
\lesssim \norm{\ang{v}^{k}F}_{L^{2}(\TT^{2}\times \RR^{2})}^{2}+\frac{1}{\delta}\norm{u}_{L^{2}(\TT^{2})}^{2}+\delta\norm{\ang{v}^{k}F_{2}}_{L^{q}(\TT^{2}\times \RR^{2})}^{q}\norm{\nabla_{u}}_{L^{2}(\TT^{2})}^{2}
\end{multline*}
where $\delta$ can be taken arbitrarily small. In order to control the quantity \\$\norm{\ang{v}^{k}F_{2}}_{L^{q}(\TT^{2}\times \RR^{2})}^{q}$ at the end of the proof we will impose that $kq< 4 +\eps$. 
On the other hand for the second term on the r.h.s. of $(V)$, introduce $\alpha > 0$ such that $\alpha p > 2$ so that
\begin{multline*}
\int_{\RR^{2}}\int_{\TT^{2}}\ang{v}^{2k} \abs{\nabla_{v}F}\,\abs{u}\, F_{1}\,dxdv 
= \int_{\RR^{2}}\int_{\TT^{2}} \ang{k}\abs{\nabla_{v}F} \frac{\abs{u}}{\ang{v}^{\alpha}}\ang{v}^{k+\alpha}F_{1}\,dxdv\\
\leq\norm{\ang{v}^{k}\abs{\nabla_{v}F}}_{L^{2}(\TT^{2}\times \RR^{2})} \norm{u}_{L^{p}(\TT^{2})} \left(\int_{\RR^{2}}\frac{1}{\ang{v}^{\alpha p}}\,dv \right)^{\frac{1}{p}} \norm{\ang{v}^{k+\alpha}F_{1}}_{L^{q}(\TT^{2}\times \RR^{2})}.
\end{multline*}
Now we apply Gagliardo-Niremberg inequality and Young inequality, in the same manner as we treated \eqref{eq:uniqueness V_1}, obtaining 
\begin{multline*}
\hspace{-0.5cm}\lesssim \delta \norm{\ang{v}^{k}\abs{\nabla_{v}F}}_{L^{2}(\TT^{2}\times \RR^{2})}^{2}\hspace{-0.1cm}+\frac{1}{\delta^{2}}\norm{u}_{L^{2}(\TT^{2})}^{2}+\delta\norm{\ang{v}^{k+\alpha}\hspace{-0.1cm}F_{1}}_{L^{q}(\TT^{2}\times \RR^{2})}^{q}\hspace{-0.1cm}\norm{\nabla{u}}_{L^{2}(\TT^{2})}^{2}.
\end{multline*}
We will then require that $(k+\alpha)q < 4+\eps$ in order to match our hypothesis on $M_{4+\eps}F(0)$. This ends the term in $(V)$. For the last one, by the product rule
\begin{multline*}
(VI) \hspace{-0.1cm}\lesssim \hspace{-0.1cm}\norm{\ang{v}^{k}F}_{L^{2}(\RR^{2}\times \TT^{2})}^{2} \hspace{-0.1cm}+ \hspace{-0.1cm}\int_{\RR^{2}}\hspace{-0.1cm}\int_{\TT^{2}}\hspace{-0.1cm}\ang{v}^{2k+1}\nabla_{v}(F^{2})\,dxdv \lesssim \norm{\ang{v}^{k}F}_{L^{2}(\RR^{2}\times \TT^{2})}^{2}\hspace{-0.1cm}.
\end{multline*}
What is left, before applying Gronwall Lemma, is only to find parameters $(k,p,q,\alpha)$ matching all the needed constraints:
\[
\begin{cases}
k>2;\\
\frac{1}{p}+\frac{1}{q}+\frac{1}{2}=1;\\
\alpha p >2;\\
(k+\alpha)q<4+\eps.
\end{cases}	
\]
The rationale behind this is the following: $k$ and $q$ can be taken arbitrarily close to $2$. Doing so, $p$ will be very large and hence $\alpha$ can be take arbitrarily small preserving the condition $\alpha p >2$, and having $(k+\alpha)$ close to $2$. 

These conditions allow us obtain that
\[
\norm{\ang{v}^{k+\alpha}F_{1}}_{L^{q}(\TT^{2}\times \RR^{2})}^{q},\norm{\ang{v}^{k}F_{2}}_{L^{q}(\TT^{2}\times \RR^{2})}^{q},\int_{\RR^{2}}\frac{1}{\ang{v}^{\alpha p}}\,dv\leq C,
\]
being
\begin{multline*}
\norm{\ang{v}^{k+\alpha}F_{1}}_{L^{q}(\TT^{2}\times \RR^{2})}^{q} \leq \norm{F_{1}}_{L^{{\infty}}(\TT^{2}\times \RR^{2})}^{q-1}\int_{\RR^{2}}\int_{\TT^{2}}\ang{v}^{(k+\alpha)q}F_{1}\,dxdv\\
\lesssim \norm{F_{1}}_{L^{{\infty}}(\TT^{2}\times \RR^{2})}^{q-1}(\norm{F_{1}}_{L^{1}(\TT^{2}\times \RR^{2})}+M_{(k+\alpha)q} F_{1}) \leq C.
\end{multline*}
Summarizing we have obtained 
\begin{multline*}
\frac{d}{dt}\norm{u_{t}}_{L^{2}(\TT^{2})}^{2} + \frac{d}{dt}\norm{\ang{v}^{k}F_{t}}_{L^{2}(\TT^{2}\times \RR^{2})}^{2} \leq	\\
\leq C_{1}\norm{u_{t}}_{L^{2}(\TT^{2})}^{2} + C_{2}\norm{\ang{v}^{k}F_{t}}_{L^{2}(\TT^{2}\times \RR^{2})}^{2},
\end{multline*}
hence by  Gronwall Lemma we obtain $u \equiv 0$ and $F \equiv 0$, proving uniqueness. 
\end{proof}

%%%%%%%%%%%%%%%%%%%%%%%%%%%%%%%%%%%%%%%%%%%%%%%%%%%%%%%%%%
\section{Scaling limit for the truncated system}\label{sec:limittruncated}
In this section we focus on the proof of a first tightness result. As remarked in the introduction we will first prove the convergence of \eqref{eq:PS-VNSR} to \eqref{eq:PDE-VNS}. To do so, we will show that, if the cutoff threshold $R$ is large enough, then the system \eqref{eq:PDE-VNSR} coincide with \eqref{eq:PDE-VNS}.
This whole section is devoted to the proof of this intermediate result:
\begin{prop}\label{prop:PSRtoVNS} 
Under hypothesis of subsection \ref{subsec:hypo} and if $R \geq \mathbf{K}_{u}+1 $, where the constant $\mathbf{K}_{u}$ will be specified later (Proposition \ref{prop:uRinfinfuniformbound}), the family of laws $\left\{ Q^{N,R}\right\}$ of the couple $(u^{N,R},S^{N,R})_{N\in\NN}$ is tight on $C([0,T]\times\TT^{2}) \times C([0,T];\PP_{1}(\TT^{2}\times \RR^{2}))$. Moreover $\left\{ Q^{N,R}\right\}_{N\in\NN}$ converges weakly to $\delta_{(u,F)}$, where the couple $(u,F)$ is the unique weak solution of system of equation \eqref{eq:PDE-VNS}.
\end{prop}
\noindent With a special argument we will be finally able to remove the cut-off also in the approximating system and get our main result, Theorem \ref{teo:PStoVNS}.

\subsection{Tightness}

In order to prove Proposition \ref{prop:PSRtoVNS} we have to establish the tightness of the laws of the empirical measure $S^{N,R}$ and that of $u^{N,R}$. We start by dealing first with the empirical measure, being the easier of the two.
The tightness of $S^{N,R}$ follows easily by a well known criterion, \cite{sznitman1991topics}, being the particles exchangeable and due to the presence of the cut-off.%and the coefficient of the SDEs in \eqref{eq:PS-VNSR} bounded. 

\begin{prop}
\label{prop:tightSN}
The family of laws $\{Q^{N,R,S}\}_{N\in\NN}$ of the empirical measure $\{S^{N,R}_{\cdot}\}_{N\in\NN}$ is relatively compact with respect of the weak convergence \\ on $C\left([0,T];\PP_{1}(\TT^{2}\times \RR^{2})\right)$.
\end{prop}

We now focus on the tightness of the laws of $u^{N,R}$. 
Notice that the coupling term that appear in the equation for $u^{N,R}$ in \eqref{eq:PS-VNSR} can be rewritten as
\begin{multline*} 
\chi_{R}(u^{N,R}_{t})\frac{1}{N}\sum_{i=1}^{N} (u^{N,R}_{\eps_{N}}(X^{i,N,R}_{t})-V^{i,N,R}_{t})\,\delta^{\eps_{N}}_{X^{i,N,R}_{t}} = \\
\chi_{R}(u^{N,R}_{t}) \int_{\RR^{2}}\int_{\TT^{2}} (u^{N,R}_{\eps_{N}}(x')-v')\theta^{0,\eps_{N}}(x-x')S^{N,R}_{t}(dx',dv')=\\
= \chi_{R}(u^{N,R}_{t})(\theta^{0,\eps_{N}}*(u^{N,R}_{\eps_{N}}-v)S^{N,R}_{t})(x).
\end{multline*}
In order to pass to the limit in the previous term, it is required that $u^{N,R}$ is converging uniformly over $\TT^{2}$, since $S^{N,R}$ is converging only weakly as probability measure. Hence, we look for a tightness criterion for $\{u^{N,R}\}_{N\in\NN}$ in $C(\TT^{2})$. By Sobolev embedding in dimension two we have $H^{2}(\TT^{2}) \hookrightarrow C(\TT^{2})$ (and also in the space of holder continuous functions). Thus, to get estimates on second derivative of $u^{N,R}$, we start by looking at the equation for $u^{N,R}$ in vorticity form: 
\begin{multline}\label{eq:vorticityPS-VNSR}
\partial_{t}\omega^{N,R} = \Delta \omega^{N,R} - u^{N,R} \cdot \nabla \omega^{N,R} \\ - \frac{\chi_{R}(u^{N,R}_{t})}{N}\sum_{i=1}^{N}\big(u^{N,R}_{\eps_{N}}(X^{i,N,R}_{t})-V^{i,N,R}_{t})\big)\nabla^{\perp} \cdot\delta^{\eps_{N}}_{X^{i,N,R}_{t}}.
\end{multline}
In order to be able to obtain a priori estimates on $\omega^{N,R}$ we need first to rewrite the coupling term in \eqref{eq:vorticityPS-VNSR} as a function of the mollified empirical measure $F^{N,R}$. We highlight that this is one of the most important key steps in this work, that remarks the importance to introduce the mollified empirical measure, and justify all the following computations. 

\begin{lem}\label{lemma:couplingrepr}
\[
\frac{1}{N}\sum_{i=1}^{N} V^{i,N,R}_{t}\delta^{\eps_{N}}_{X^{i,N,R}_{t}}(x) = \int_{\RR^{2}}v F^{N,R}_{t}(x,v)\,dv = m_{1}F^{N,R}_{t}(x)
\]
\[
\frac{1}{N}\sum_{i=1}^{N} \delta^{\eps_{N}}_{X^{i,N,R}_{t}}(x) = \int_{\RR^{2}}F^{N,R}_{t}(x,v)\,dv = m_{0}F^{N,R}_{t}(x)
\]
\end{lem}
\begin{proof}
Notice that
\begin{multline*}
\frac{1}{N}\sum_{i=1}^{N} V^{i,N,R}_{t}\delta^{\eps_{N}}_{X^{i,N,R}_{t}}(x) = \int_{\RR^{2}}\int_{\TT^{2}}\theta^{0,\eps_{N}}(x-x')v'S^{N,R}_{t}(dx',dv') \\
= \int_{\RR^{2}}\int_{\RR^{2}}\int_{\TT^{2}}\theta^{0,\eps_{N}}(x-x')\theta^{1,\eps_{N}}(v-v')v'S^{N,R}_{t}(dx',dv')\,dv,
\end{multline*}
and
\begin{multline*}
\int_{\RR^{2}}vF^{N,R}_{t}(x,v)\,dv = \int_{\RR^{2}}\int_{\RR^{2}}\int_{\TT^{2}}\theta^{0,\eps_{N}}(x-x')\theta^{1,\eps_{N}}(v-v')vS^{N,R}_{t}(dx',dv')dv
\end{multline*}
so that to complete the proof we only need to prove 
\[
\int_{\RR^{2}}\int_{\RR^{2}}\int_{\TT^{2}}\theta^{0,\eps_{N}}(x-x')\theta^{1,\eps_{N}}(v-v')(v-v')S^{N,R}_{t}(dx',dv')dv = 0.
\]
However this is true due to
\[
\int_{\RR^{2}}\theta^{1,\eps_{N}}(v-v')(v-v')\,dv = 0
\]
by the hypothesis of symmetry \eqref{hypo:mollifier} in \ref{subsec:hypo}. The second identity of the Lemma follows by the very definition of $\delta^{\eps_{N}}_{X^{i,N,R}_{t}}$. This ends the proof.
\end{proof}
As stated above, we look for an estimate in $H^{2}(\TT^{2})$ for $u^{N}$. This are obtained by energy type estimates  for the fluid in the vorticity form.
\begin{lem}\label{lemma:omegaL2W12}
\begin{multline*}
\EE{\sup_{t\in[0,T]}\norm{\omega^{N,R}_{t}}_{L^{2}(\TT^{2})}^{2} + \int_{0}^{T}\norm{\nabla \omega_{s}^{N,R}}_{L^{2}(\TT^{2})}^{2}\,ds } \lesssim \EE{\norm{\omega^{N,R}_{0}}_{L^{2}(\TT^{2})}^{2}} \\ + \EE{\norm{m_{1}F^{N,R}}_{L^{2}([0,T]\times\TT^{2})}^{2}}
+ R\EE{\norm{m_{0}F^{N,R}}_{L^{2}([0,T]\times\TT^{2})}^{2}}.
\end{multline*}
\end{lem}
\begin{proof}
The thesis follows by classical energy inequality for $\omega^{N,R}$ and by using lemma \ref{lemma:couplingrepr}.
\end{proof}
\noindent We remark that the previous computation was only possible due to the presence of the cuf-off. Without it, it would not be possible to decouple the fluid and particles in \eqref{eq:PS-VNSR}, hence permitting us to closing estimates for fluid and particles separately. 

\noindent From Lemma \ref{lemma:omegaL2W12} it is clear that it is necessary to control the $L^{2}$ norm of both $m_{1}F^{N,R}$ and $m_{0}F^{N,R}$. To do so we will exploit Lemma \ref{lemma:marginalinequ} and thus look for an estimate for $M_{6}F^{N,R}$ and for $(F^{N,R})^{4}$. This is exactly the goal of the next lemmas.

\begin{lem}
\label{lemma:FNL4}
There exists a constant $C_{T,R,4}$, independent on $N$, such that
\[
\sup_{t\in[0,T]}\EE{\norm{F^{N,R}_{t}}_{L^{4}(\TT^{2}\times\RR^{2})}^{4}} \leq C_{T,R,4}.
\]
\end{lem} 
\begin{proof}
This proof strictly follows the proof of Lemma 3.3 in \cite{oldVNS}.
By It\^o formula and integration by parts we have 

\[
\frac{1}{4}d\int_{\RR^{2}}\int_{\TT^{2}} (F^{N,R}_{t})^{4}\,dxdv + \frac{3\sigma^{2}}{2}\int_{\RR^{2}}\int_{\TT^{2}}(F^{N,R}_{t})^{2}\abs{\nabla_{v}F^{N,R}_{t}}^{2}\,dxdvdt =
\]

\begin{equation}\label{eq:FN4,divx}
-\int_{\RR^{2}}\int_{\TT^{2}}(F^{N,R}_{t})^{3}\div_{x}(\theta^{\eps_{N}}*(vS^{N,R}_{t}))\,dxdvdt
\end{equation}
\begin{equation}
\label{eq:FN4,divv}
- \int_{\RR^{2}}\int_{\TT^{2}}(F^{N,R}_{t})^{3}\div_{v}(\theta^{\eps_{N}}*(u^{N,R}_{\eps_{N}}(t,x)\chi_{R}(u^{N,R}_{t})-v)S^{N,R}_{t})\,dxdvdt
\end{equation}
\begin{equation}\label{eq:FN4,martingales}
+ \int_{\RR^{2}}\int_{\TT^{2}}(F^{N,R}_{t})^{3}\,dM^{N,\eps_{N}}_{t}\,dxdv + \int_{\RR^{2}}\int_{\TT^{2}}(F^{N,R}_{t})^{2}\,d[M^{N,\eps_{N}}]_{t}\,dxdv.
\end{equation}
We estimate each of the terms above separately. Concerning \eqref{eq:FN4,divx}, we can rewrite the convolution inside the integral as
\[
\div_{x}(\theta^{\eps_{N}}*(vS^{N,R}_{t})) = v \cdot \nabla_{x}(\theta^{\eps_{N}}*S^{N,R}_{t})-((\nabla_{x}\theta^{\eps_{N}}\cdot v)*S^{N,R}_{t}).
\]
Hence, for the first term on the r.h.s. we have
\[
-\int_{\RR^{2}}\int_{\TT^{2}}(F^{N,R}_{t})^{3}\nabla_{x}F^{N,R}_{t}\cdot v\,dx\,dv\,dt = \int_{\RR^{2}}\int_{\TT^{2}} \nabla_{x}(F^{N,R}_{t})^{4} \cdot v\,dx\,dv\,dt = 0.
\]
For the second one, note that due to our hypothesis on the mollifiers $\theta^{0}(x)$ and $\theta^{1}(v)$ we have
\[
 \abs{\nabla_{x}\theta^{0,\eps_{N}}(x-x')}\theta^{1,\eps_{N}}(v-v')\abs{(v-v')}
\]
\[
= \eps_{N}^{-1}\eps_{N}^{-d}\abs{\nabla_{x}\theta^{0}(\eps_{N}^{-1}(x-x'))}\eps_{N}^{-d}\theta^{1}(\eps_{N}^{-1}(v-v'))\abs{v-v'}
\]
\[
\leq \eps_{N}^{-d}\abs{\theta^{0}(\eps_{N}^{-1}(x-x'))}\eps_{N}^{-d}\theta^{1}(\eps_{N}^{-1}(v-v'))\frac{\abs{v-v'}}{\eps_{N}}
\]
\[
\leq \theta^{0,\eps_{N}}(x-x')\theta^{1,\eps_{N}}(v-v')2
\]
implying 
\[
\abs{\eqref{eq:FN4,divx}} \lesssim \norm{F^{N,R}_{t}}_{L^{4}(\RR^{2}\times \TT^{2})}^{4}.
\]
The main differences with respect to the proof of \cite{oldVNS} concerns the term in \eqref{eq:FN4,divv}: 
 we split it into the part containing $u$ and the one with the velocity variable:
the first one follows easily by the truncation, being
\[
\abs{\int_{\RR^{2}}\int_{\TT^{2}}(F^{N,R}_{t})^{3}\div_{v}(\theta^{\eps_{N}}*u^{N,R}_{\eps_{N}}(t,x)\chi_{R}(u^{N,R}_{t})S^{N,R}_{t})\,dxdv} 
\]
\[
= \abs{\int_{\RR^{2}}\int_{\TT^{2}}\nabla_{v}(F^{N,R}_{t})^{3}(\theta^{\eps_{N}}*u^{N,R}_{\eps_{N}}(t,x)\chi_{R}(u^{N,R}_{t})S^{N,R}_{t})\,dxdv}
\]
\[
\leq \int_{\RR^{2}}\int_{\TT^{2}}\abs{\nabla_{v}(F^{N,R}_{t})^{3}}(\theta^{\eps_{N}}*\abs{u^{N,R}_{\eps_{N}}(t,x)\chi_{R}(u^{N,R}_{t})}S^{N,R}_{t})\,dxdv 
\]
\begin{multline*}
\leq R\int_{\RR^{2}}\int_{\TT^{2}}\abs{\nabla_{v}F^{N,R}_{t} F^{N,R}_{t}}(F^{N,R}_{t})^{2}\,dx\,dv \lesssim \frac{1}{\delta}\norm{F^{N,R}_{t}}_{L^{4}(\TT^{2}\times \RR^{2})}^{4} +\\ \delta\int_{\RR^{2}}\int_{\TT^{2}}(F^{N,R}_{t})^{2}\abs{\nabla_{v}F^{N,R}_{t}}^{2}\,dxdv
\end{multline*}
and by choosing $\delta$ small enough we can take the second term to the l.h.s. maintaining the positivity. For the other one we again split it into a basic term plus a commutator
\begin{multline}\label{eq:splitV}
\int_{\RR^{2}}\int_{\TT^{2}}(F^{N,R}_{t})^{3}\div_{v}(\theta^{\eps_{N}}*vS^{N,R}_{t})\,dxdv = \\\int_{\RR^{2}}\int_{\TT^{2}}(F^{N,R}_{t})^{3}\div_{v}(v(\theta^{\eps_{N}}*S^{N,R}_{t}))\,dxdv\\
- \int_{\RR^{2}}\int_{\TT^{2}}(F^{N,R}_{t})^{3}\div_{v}(\theta^{\eps_{N}}v*S^{N,R}_{t})\,dxdv.
\end{multline}
The first term on the r.h.s. on \eqref{eq:splitV} is easily handled by direct computation
\begin{multline*}
=  -\int_{\RR^{2}}\int_{\TT^{2}} \nabla_{v} (F^{N,R}_{t})^{3} \cdot v \,F^{N,R}_{t}\,dxdv =\\ -\frac{1}{4} \int_{\RR^{2}}\int_{\TT^{2}} \nabla_{v} (F^{N,R}_{t})^{4} \cdot v \,dxdv = \frac{1}{2}\norm{F^{N,R}_{t}}_{L^4(\TT^{2}\times \RR^{2})}^{4},
\end{multline*}
while the second one is more tricky: we compute the divergence on $v$ and obtain
\begin{multline*}	
 \int_{\RR^{2}}\int_{\TT^{2}}(F^{N,R}_{t})^{3}\div_{v}(\theta^{\eps_{N}}v*S^{N,R}_{t})\,dxdv = 
 2 \int_{\RR^{2}}\int_{\TT^{2}} (F^{N,R}_{t})^{4}\, dxdv \\
 \int_{\RR^{2}}\int_{\TT^{2}} (F^{N,R})^{3}\int_{\RR^{2}}\int_{\TT^{2}} \theta^{0,\eps_{N}}(x-x')\nabla_{v}\theta^{1,\eps_{N}}(v-v')\cdot (v-v')\,S^{N,R}_{t}(dx',dv')dxdvdt  \\
\hspace{-8cm}\leq 2\norm{F^{N,R}_{t}}_{L^4(\TT^{2}\times \RR^{2})}^{4}+\\
\int_{\RR^{2}}\hspace{-0.1cm}\int_{\TT^{2}} \hspace{-0.1cm}\abs{\nabla(F^{N,R})^{3}}\int_{\RR^{2}}\hspace{-0.1cm}\int_{\TT^{2}} \hspace{-0.2cm}\theta^{0,\eps_{N}}(x-x')\theta^{1,\eps_{N}}(v-v') \abs{v-v'}\hspace{-0.1cm}S^{N,R}_{t}(dx',dv')dxdvdt.
\end{multline*}
Now we just look at the most inner term in the last inequality: using the compact support assumption  for $\theta^{1}(v)$, see \ref{subsec:hypo} hypothesis \eqref{hypo:mollifier}, we get
\[
\theta^{0,\eps_{N}}(x-x')\theta^{1,\eps_{N}}(v-v')\abs{v-v'} \leq 2\eps_{N}\theta^{0,\eps_{N}}(x-x')\theta^{1,\eps_{N}}(v-v'),
\]
which leads to $\eqref{eq:splitV}$ being
\begin{multline*}
\eqref{eq:splitV}\lesssim \norm{F^{N,R}_{t}}_{L^4(\TT^{2}\times \RR^{2})}^{4} + \eps_{N}\int_{\RR^{2}}\int_{\TT^{2}}\abs{\nabla(F^{N,R}_{t})^{3}} F^{N,R}_{t}\,dxdv \\
\lesssim \norm{F^{N,R}_{t}}_{L^4(\TT^{2}\times \RR^{2})}^{4}+ 2\eps_{N}\int_{\RR^{2}}\int_{\TT^{2}}\abs{\nabla_{v}F^{N,R}_{t}}^{2}(F^{N,R}_{t})^{2}\,dxdv.
\end{multline*}
We now deal with the two last term in \eqref{eq:FN4,martingales}: the integral with respect to the martingale $M^{N,\eps_{N}}_{t}$ will vanish when computing the expected value, while for the integral with respect to the quadratic variation we have
\begin{multline*}
\int_{\RR^{2}}\int_{\Pi^{d}}(F^{N}_{t})^{2}\,d[M^{N,\eps_{N}}]_{t}\,dxdv = \frac{\sigma^{2}}{N}\int_{\RR^{2}}\int_{\Pi^{d}} (F^{N}_{t})^{2}(\abs{\nabla_{v}\theta^{\eps_{N}}}^{2}*S^{N}_{t})\,dxdvdt\leq \\
\sigma^{4}\norm{F^{N}_{t}}_{L^{4}}^{4}dt + \frac{1}{N^{2}}\int_{\RR^{2}}\int_{\Pi^{d}}(\abs{\nabla_{v}\theta^{\eps_{N}}}^{2}*S^{N}_{t})^{2}\,dxdvdt.
\end{multline*}
The square outside the convolution $(\abs{\nabla_{v}\theta^{\eps_{N}}}^{2}*S^{N}_{t})^{2}$ can be troublesome, but we can handle it using the property of compact support of $\theta^{1}(v)$ and the separation of variables, in the following way:
\begin{multline*}
\int_{\RR^{2}}\int_{\TT^{2}}(\abs{\nabla_{v}\theta^{\eps_{N}}}^{2}*S^{N}_{t})^{2}\,dx\,dv \lesssim \\ \frac{1}{N}\sum_{i=1}^{N} \left(\int_{\RR^{2}}\int_{\TT^{2}}\abs{\nabla_{v}\theta^{1,\eps_{N}}(v-V^{i,N,R}_{t})}^{2}\theta^{0,\eps_{N}}(x-X^{i,N,R}_{t})^{2}\,dx\,dv\right)^{2}
\end{multline*}
\[
\lesssim \frac{1}{N}\sum_{i=1}^{N}\int_{\RR^{2}}\abs{\nabla_{v}\theta^{1,\eps_{N}}(v-V^{i,N,R}_{t})}^{4}\,dv \int_{\TT^{2}}\theta^{0,\eps_{N}}(x-X^{i,N,R}_{t})^{4}\,dx.
\]
Now we compute
\[
\int_{\RR^{2}}\abs{\nabla_{v}\theta^{1,\eps_{N}}(v-V^{i,N}_{t})}^{4}\,dv = C N^{5\beta},
\]
\[ 
\int_{\Pi^{d}}\theta^{0,\eps_{N}}(x-X^{i,N}_{t})^{4}\,dx = CN^{3\beta},
\]
e substitute into the integral for the quadratic variation
\[
\frac{1}{N^{2}}\int_{\RR^{2}}\int_{\Pi^{d}}(\abs{\nabla_{v}\theta^{\eps_{N}}}^{2}*S^{N}_{t})^{2}\,dxdv \lesssim \frac{1}{N^{2}} N^{5\beta} N^{3\beta}
\]
which is bounded for $\beta \leq \frac{1}{4}$.

Summarizing we have obtained
\[
d\norm{F^{N,R}_{t}}_{L^4(\TT^{2}\times \RR^{2})}^{4} + C\int_{\RR^{2}}\int_{\TT^{2}}(F^{N,R}_{t})^{2}\abs{\nabla_{v}F^{N,R}_{t}}^{2}\,dxdvdt \leq 
\]
\[
\lesssim C_{R}\norm{F^{N,R}_{t}}_{L^4(\TT^{2}\times \RR^{2})}^{4}\,dt+ \int_{\RR^{2}}\int_{\TT^{2}}(F^{N,R}_{t})^{3}\,dM^{N,\eps_{N}}_{t}\,dxdv + C dt
\]
which, after taking the average, ends the proof by standard Gronwall lemma.
\end{proof}
\noindent By interpolation between $L^{p}$ spaces, and the fact that $F^{N,R}_{t}$ is a probability density function, we obtain the following corollary:
\begin{cor}\label{cor:FNL2}
There exists a constant $C_{T,R,2}$, independent on $N$, such that
\[
\sup_{t\in[0,T]}\EE{\norm{F^{N,R}_{t}}_{L^{2}(\TT^{2}\times\RR^{2})}^{2}} \leq C_{T,R,2}.
\]
\end{cor}

We now proceed to bound the moments on the $v$-component of the mollified empirical measure $F^{N,R}$.
The proof of the next Lemma follows by the very definition of $M_{k}F^{N,R}$ by using change of variable formula. 
% by direct computation and can be found in \cite{oldVNS}, Lemma 3.4. 
\begin{lem}\label{lemma:M_k}
For all $k\leq 6$ and for all $N$ and $R$, there exists a constant $C_{k}^{T,R}$, independent on $N$ such that
\[
\EE{\sup_{t \in [0,T]} M_{k}F^{N,R}_{t}}\leq C_{k}^{T,R}.
\]
\end{lem}
\begin{proof}
The proof follows by expanding $F^{N,R}$ as a summation, and by a change of variables inside the integral with respect to $v$. This allow to bound the $k$-th moments along $v$ of $F^{N,R}$ by
\[
\EE{\sup_{t\in[0,T]}\abs{V^{i,N,R}_{t}}^{k}}.
\]
Moreover, we can bound the expected value in the previous formula using the SDEs for the particles velocity, by using the truncation and the hypothesis on the initial conditions.
\end{proof}
\noindent Summarizing, up to this point we were able to prove the following bounds, independently on $N$:
\[
\sup_{t\in[0,T]}\EE{\norm{m_{0}F^{N,R}_{t}}_{L^{2}(\TT^{2})}^{2}} \leq C_{T,R},
\]
\[
\sup_{t\in[0,T]}\EE{\norm{m_{1}F^{N,R}_{t}}_{L^{2}(\TT^{2})}^{2}} \leq C_{T,R},
\]
by Lemmas  \ref{lemma:FNL4}, \ref{lemma:M_k} and inequality 3. and 4. from Lemma \ref{lemma:marginalinequ}. Also
\[
\EE{\sup_{t\in[0,T]}\norm{\omega^{N,R}_{t}}_{L^{2}(\TT^{2})}^{2} + \int_{0}^{T}\norm{\nabla \omega_{s}^{N,R}}_{L^{2}(\TT^{2})}^{2}\,ds }\leq C_{T,R}.
\]
% \[
% \EE{\norm{\omega^{N,R}}_{L^{2}([0,T];H^{1}(\TT^{2}))}^{2}} \leq C_{T,R}.
% \]
by Lemma \ref{lemma:omegaL2W12}.\\
Hence we have obtained the desired bound for the fluid in vorticity form. However, in order to obtain convergence, we need to apply an appropriate tightness criterion. \\
Classical Aubin-Lions Lemma states that when $E_{0}\subseteq E\subseteq E_{1}$ are three Banach spaces with continuous embedding, and $E_{0}$ compactly embedded into $E$, then for all $p,q < \infty$ the space
$L^{p}([0,T];E_{0}) \cap W^{1,q}([0,T];E_{1})$ is compactly embedded into $L^{p}([0,T];E)$.
Hence, we can apply this criterion choosing $p=q=2$ and 
$E_{0} = H^{2}(\TT^{2})$, $E = C(\TT^{2})$ and $E_{1} = H^{-1}(\TT^{2})$ to obtain 
\[
L^{2}([0,T];H^{2}(\TT^{2}))\cap W^{1,2}([0,T];H^{-1}(\TT^{2}))\hookrightarrow L^{2}([0,T];C(\TT^{2}))
\]
and the embedding is compact. Thus, in order to obtain the required tightness result, we also need an a priori estimate for the time derivative of $\omega^{N,R}$: 
\begin{lem}\label{lemma:omegaH^-1}
For every $\eps > 0$ there exists $Z>0$, such that
\[
\PP \left( \norm{\omega^{N,R}}_{W^{1,2}([0,T];H^{-1}(\TT^{2}))} > Z\right)\leq \eps
\]
\end{lem}
\begin{proof}
By Lemma \ref{lemma:omegaL2W12} we already have the result for the norm of $\omega^{N}$ in the space $L^{2}([0,T];L^{2}(\TT^{2}))$. Since $H^{1}\hookrightarrow L^{2}\hookrightarrow H^{-1}$ we already know that
\[
\PP \left( \norm{\omega^{N,R}}_{L^{2}([0,T];H^{-1}(\TT^{2}))} > Z\right)\leq \eps.
\]
Hence we only need to estimate $\norm{\partial_{t}\omega^{N,R}}_{L^{2}([0,T];H^{-1}(\TT^{2}))}$. Thus we compute the $H^{-1}$ norm both sides in the equation for $\omega^{N,R}$, obtaining 
\begin{multline*}
\norm{\partial_{t}\omega^{N,R}_{t}}_{H^{-1}(\TT^{2})} \lesssim \norm{\Delta \omega^{N,R}_{t}}_{H^{-1}(\TT^{2})} + \norm{u^{N,R}_{t}\cdot \nabla\omega^{N,R}_{t}}_{H^{-1}(\TT^{2})}\\ + R\norm{m_{0}F^{N,R}_{t}}_{L^{2}(\TT^{2})}+\norm{m_{1}F^{N,R}_{t}}_{L^{2}(\TT^{2})}
\end{multline*}
by classical argument and integration by parts. Taking the square and integrating both sides we obtain
\begin{multline*}
\int_{0}^{T}\norm{\partial_{t}\omega^{N,R}_{t}}_{H^{-1}(\TT^{2})}^{2}\,dt  \lesssim \int_{0}^{T}\norm{\nabla \omega^{N,R}_{t}}_{L^{2}(\TT^{2})}^{2}\,dt + \\
\sup_{t\in[0,T]}\norm{\omega^{N,R}_{t}}_{L^{2}(\TT^{2})}^{2}\int_{0}^{T}\norm{u^{N,R}_{t}}_{C(\TT^{2})}^{2}\,dt\\ + R\int_{0}^{T}\norm{m_{0}F^{N,R}_{t}}_{L^{2}(\TT^{2})}^{2}\,dt +\int_{0}^{T}\norm{m_{1}F^{N,R}_{t}}_{L^{2}(\TT^{2})}^{2}\,dt.
\end{multline*}
Finally, We compute probability both sides 
\[
\PP \left(\int_{0}^{T}\norm{\partial_{t}\omega^{N,R}_{t}}_{H^{-1}(\TT^{2})}^{2}\,dt  > R \right)
\]
and use the fact that we can split product term inside probabilities
\begin{multline*}
\PP \left(\sup_{t\in[0,T]}\norm{\omega^{N,R}_{t}}_{L^{2}(\TT^{2})}^{2}\int_{0}^{T}\norm{u^{N,R}_{t}}_{C(\TT^{2})}^{2}\,dt> R \right) \\
\leq \PP \left(\sup_{t\in[0,T]}\norm{\omega^{N,R}_{t}}_{L^{2}(\TT^{2})}^{2}> \sqrt{R} \right) + \PP \left(\int_{0}^{T}\norm{u^{N,R}_{t}}_{C(\TT^{2})}^{2}\,dt> \sqrt{R} \right).
\end{multline*}
Since all the terms above are bounded in expected value, we can apply Chebyshev inequality to make each term smaller than $\eps$. This ends the proof.
\end{proof}
At this point, thanks to Aubin's Lemma,  we are able to obtain a first tightness result for the law of $u^{N,R}$  in $L^{2}([0,T];C(\TT^{2}))$. However, while this is enough to prove a convergence result, as partially done in \cite{oldVNS}, by having only $L^{2}$ estimates on time we won't be able to remove the cutoff at the particle level, thus obtaining Theorem \ref{teo:PStoVNS}. Hence we will have to improve our estimates in order to obtain stronger time convergence. 
We apply Corollary 8 in \cite{simon1986compact} by taking
\[
X=H^{1+2\alpha}(\TT^{2}),\quad B=H^{1+2\alpha-\eps}(\TT^{2}),\quad Y=H^{-1}(\TT^{2}),
\]
where $\eps < 2\alpha$ and where $X\hookrightarrow Y$ is compact. The interpolation inequality between the space $B$ and $X,Y$, required in Corollary 8,  it is an easy result of Fourier analysis since we are on the torus. Hence we have that
\[
L^{\infty}([0,T];H^{1+2\alpha}(\TT^{2}))\cap W^{1,2}([0,T];H^{-1}(\TT^{2}))\hookrightarrow C([0,T];H^{1+2\alpha-\eps}(\TT^{2})) 
\]
with a compact embedding. Hence, by Sobolev embedding in dimension two of $H^{1+2\alpha-\eps}(\TT^{2})$ into $C(\TT^{2})$ we also have that
\[
L^{\infty}([0,T];H^{1+2\alpha}(\TT^{2}))\cap W^{1,2}([0,T];H^{-1}(\TT^{2}))\hookrightarrow C([0,T]\times \TT^{2}) 
\]
with a compact embedding.
The result of course also holds when replacing $H^{1+2\alpha}(\TT^{2})$ with $H^{2}(\TT^{2})$.
However we were not able to obtain a uniform in time result for the $H^{2}$ norm and hence we tried to trim our requirements. 
To do so, we first rewrite the equation for $\omega^{N,R}$ in its mild formulation

\begin{multline}\label{eq:vorticity_mild}
\omega^{N,R}_t=e^{t\Delta}\omega^{N,R}_0 -\int_{0}^{t} e^{(t-s)\Delta}u^{N,R}_s\cdot \nabla \omega^{N,R}_s ds\\
-\int_{0}^{t} e^{(t-s)\Delta}\nabla^\perp\cdot \frac{1}{N}\sum_{i=1}^N(u^{N,R}_{\eps_N}(X^{i,N,R}_s)\chi_{R}(u^{N,R}_{s})-V^{i,N,R}_s) \delta^{\eps_{N}}_{X^{i,N,R}_{s}}ds.
\end{multline}
\begin{lem}\label{lemma:tightnessmild}
For all $\alpha < \frac{1}{2}$ and for each $\eps$, there exists $Z$ such that
\[\PP \left( \left|\left|u^{N,R} \right|\right|_{L^{\infty}([0,T],H^{{1+2\alpha}})}>Z\right)\leq \eps\]
\end{lem} 
\begin{proof}
We will apply a generalized Gronwall Lemma to the function of the only time variable $\norm{u^{{N,R}}_{t}}_{H^{1+2\alpha}(\TT^{2})}$. Since $\norm{u^{{N,R}}_{t}}_{H^{1+2\alpha}(\TT^{2})}\sim \norm{\omega^{{N,R}}_{t}}_{H^{2\alpha}(\TT^{2})}$ we start by applying the operator $(I-\Delta)^{\alpha}$ on the mild formulation of vorticity equation \eqref{eq:vorticity_mild},  obtaining 
\begin{multline}\label{eq:tightnessmild1}
\left|\left| (I-\Delta)^{\alpha} \omega^{N,R}_{t}\right|\right|_{L^{2}(\TT^{2})} \leq \left|\left| (I-\Delta)^{\alpha} e^{t\Delta}\omega^{N,R}_{0}\right|\right|_{L^{2}(\TT^{2})} +\\
\int_{0}^{t} \norm{ (I-\Delta)^{\alpha}e^{(t-s)\Delta}\nabla^\perp\cdot \frac{1}{N}\sum_{i=1}^N(u^{N,R}_{\eps_N}(X^{i,N,R}_s)\chi_{R}(u^{N,R}_{s})-V^{i,N,R}_s) \delta^{\eps_{N}}_{X^{i,N,R}_{s}}}_{L^{2}(\TT^{2})}\hspace{-0.9cm} ds\\
+\int_{0}^{t}\hspace{-0.1cm} \norm{(I-\Delta)^{\alpha}e^{(t-s)\Delta}u^{N,R}_s\cdot \nabla \omega^{N,R}_s }_{L^{2}(\TT^{2})} ds\\
\end{multline}
We start by estimating the initial conditions:
\begin{multline}\label{initial_condition}
\left|\left| (I-\Delta)^{\alpha} \omega^{N,R}_{0}\right|\right|_{L^{2}(\TT^{2})}\leq\\ \left|\left| e^{t\Delta}\right|\right|_{L^{2}(\TT^{2})\to L^{2}(\TT^{2})}\left|\left| (I-\Delta)^{\alpha}\omega^{N,R}_{0}\right|\right|_{L^{2}(\TT^{2})}\lesssim \norm{\omega^{N,R}_{0}}_{H^{2\alpha}(\TT^{2})}.
\end{multline}
Regarding the second term of the r.h.s. of \eqref{eq:tightnessmild1}
\begin{multline}\label{2}
\left|\left| (I-\Delta)^{\alpha} e^{(t-s)\Delta}\nabla^\perp\hspace{-0.1cm}\cdot\hspace{-0.1cm} \frac{1}{N}\sum_{i=1}^N(u^{N,R}_{\eps_N}(X^{i,{N,R}}_s)\chi_{R}(u^{N,R}_{s})-V^{i,{N,R}}_s)\delta^{\eps_{N}}_{X^{i,N,R}_{s}} \right|\right|_{L^{2}(\TT^{2})} \\
\leq \norm{ \nabla(I-\Delta)^{-1/2}}_{L^{2}(\TT^{2})\to L^{2}(\TT^{2})}\norm{(I-\Delta)^{\alpha+1/2}e^{(t-s)\Delta}}_{L^{2}(\TT^{2})\to L^{2}(\TT^{2})}\\
\times \norm{ \frac{1}{N}\sum_{i=1}^N(u^{N,R}_{\eps_N}(X^{i,{N,R}}_s)\chi_{R}(u^{N,R}_{s})-V^{i,{N,R}}_s)\delta^{\eps_{N}}_{X^{i,N,R}_{s}}}_{L^{2}(\TT^{2})}\\
\leq \frac{C}{(t-s)^{\alpha+1/2}}\norm{ \frac{1}{N}\sum_{i=1}^N(u^{N,R}_{\eps_N}(X^{i,{N,R}}_s)\chi_{R}(u^{N,R}_{s})-V^{i,{N,R}}_s)\delta^{\eps_{N}}_{X^{i,N,R}_{s}}}_{L^{2}(\TT^{2})}\\
\leq \frac{C}{(t-s)^{\alpha+1/2}}\left(R\norm{m_{0}F^{N,R}_{t}}_{L^{2}(\TT^{2})}+\norm{m_{1}F^{{N,R}}_{t}}_{L^{2}(\TT^{2})}\right),
\end{multline}
while for the last one of \eqref{eq:tightnessmild1} we have
\begin{align*}
\Big|\Big|(&I-\Delta)^{\alpha}  e^{(t-s)\Delta} u^{N,R}_{s}\cdot \nabla \omega^{N,R}_{s}  \Big|\Big|_{L^{2}(\TT^{2})} \\
% &= \norm{(I-\Delta)^{-1/2}(I-\Delta)^{\alpha+1/2}e^{(t-s)\Delta}u^{N,R}_s\cdot \nabla \omega^{N,R}_s}_{L^{2}(\TT^{2})} \\ 
&\leq \norm{(I-\Delta)^{\alpha+1/2}e^{(t-s)\Delta}}_{L^{2}(\TT^{2})\rightarrow L^{2}(\TT^{2})} \norm{(I-\Delta)^{-1/2} u^{N,R}_s\cdot  \nabla \omega^{N,R}_s}_{L^{2}(\TT^{2})} \\
&\leq \frac{C}{(t-s)^{\alpha+1/2}} \norm{(I-\Delta)^{-1/2} u^{N,R}_s\cdot  \nabla \omega^{N,R}_s}_{L^{2}(\TT^{2})},
\end{align*}
and
\begin{equation*}
\norm{(I-\Delta)^{-1/2} u^{N,R}_s\cdot  \nabla\omega^{N,R}_s}_{L^{2}(\TT^{2})}=\sup_{\phi \in {L^{2}(\TT^{2})}}\left| \langle (I-\Delta)^{-1/2} u^{N,R}_s\cdot  \nabla\omega^{N,R}_s, \phi \rangle\right|.
\end{equation*}
Now, notice that
\begin{multline}\label{3}
\langle  u^{N,R}_s\cdot  \nabla\omega^{N,R}_s, (I-\Delta)^{-1/2} \phi \rangle = -\langle   \omega^{N,R}_s, u^{N,R}_s\cdot \nabla (I-\Delta)^{-1/2} \phi \rangle \\
\leq  \sup_{\norm{\phi}_{L^{2}(\TT^{2})}\leq 1}	\norm{\phi}_{L^{2}(\TT^{2})}\norm{u^{{N,R}}_{s}}_{\infty}\norm{\omega^{{N,R}}}_{L^{2}(\TT^{2})} .
\end{multline}
Combining \eqref{initial_condition},\eqref{2},\eqref{3}:
\begin{multline*}
\norm{\omega^{{N,R}}_{t}}_{H^{2\alpha}(\TT^{2})}\lesssim  \norm{\omega^{{N,R}}_{0}}_{H^{2\alpha}(\TT^{2})}  \\
+\int_{0}^{t}\frac{\left(R\norm{m_{0}F^{N,R}_{s}}_{L^{2}(\TT^{2})}+\norm{m_{1}F^{{N,R}}_{s}}_{L^{2}(\TT^{2})}\right)}{(t-s)^{\alpha+1/2}} ds \\
+\int_{0}^{t} \frac{\norm{u^{{N,R}}_{s}}_{L^{\infty}(\TT^{2})}\norm{\omega^{{N,R}}_{s}}_{L^{2}(\TT^{2})}}{(t-s)^{\alpha+1/2}}ds\\
\leq  C\norm{\omega^{{N,R}}_{0}}_{H^{2\alpha}(\TT^{2})} + \int_{0}^{T}\frac{\left(R\norm{m_{0}F^{N,R}_{s}}_{L^{2}(\TT^{2})}+\norm{m_{1}F^{{N,R}}_{s}}_{L^{2}(\TT^{2})}\right)}{(T-s)^{\alpha+1/2}} ds +\\
+C\left(\sup_{t\in [0,T]}\norm{\omega^{{N,R}}_{t}}_{L^{2}(\TT^{2})}\right)\int_{0}^{t} \frac{\norm{u^{{N,R}}_{s}}_{H^{1+2\alpha}(\TT^{2})}}{(t-s)^{\alpha+1/2}}ds.
\end{multline*}
 
Finally,
\begin{multline*}
\norm{u^{{N,R}}_{t}}_{H^{1+2\alpha}(\TT^{2})} \lesssim  C\norm{\omega^{{N,R}}_{0}}_{H^{2\alpha}(\TT^{2})} + \int_{0}^{T}\frac{R\norm{m_{0}F^{N,R}_{s}}_{L^{2}(\TT^{2})}}{(T-s)^{\alpha+1/2}}\,ds +\\
+\int_{0}^{T}\hspace{-0.1cm}\frac{\norm{m_{1}F^{{N,R}}_{s}}_{L^{2}(\TT^{2})}}{(T-s)^{\alpha+1/2}}\,ds +\left(\sup_{t\in [0,T]}\norm{\omega^{{N,R}}_{t}}_{L^{2}(\TT^{2})}\right)\hspace{-0.1cm}\int_{0}^{t}\hspace{-0.1cm} \frac{\norm{u^{{N,R}}_{s}}_{H^{1+2\alpha}(\TT^{2})}}{(t-s)^{\alpha+1/2}}ds.
\end{multline*}

Notice that the object above are random (for simplicity we have omitted $\omega\in \Omega$). Introduce, to short the notation, the random function 
\[
\phi(t) := \norm{u^{{N,R}}_{t}}_{H^{1+2\alpha}(\TT^{2})}.
\]
We have proved that the function $\phi$ satisfies
\[
\phi(t) \leq X_{1}+X_{2}\int_{0}^{t}\frac{\phi(s)}{(t-s)^{\alpha+1/2}}\,ds
\]
where
\[
X_{1}=\norm{\omega^{{N,R}}_{0}}_{H^{2\alpha}(\TT^{2})}+ \int_{0}^{T}\frac{\left(R\norm{m_{0}F^{N,R}_{s}}_{L^{2}(\TT^{2})}+\norm{m_{1}F^{{N,R}}_{s}}_{L^{2}(\TT^{2})}\right)}{(T-s)^{\alpha+1/2}}\,ds
\]
\[
X_{2}=\sup_{t\in [0,T]}\norm{\omega^{{N,R}}_{t}}_{L^{2}(\TT^{2})}.
\]
Notice that, by the uniform estimates proved in this section, there exist two constant $C_{1}$ and $C_{2}$, independent on $N$, such that 
\[
\EE{X_{1}}\leq C_{1}, \quad \EE{X_{2}}\leq C_{2},
\]
so that, for fixed $\eps$ we can chose $R_{1},R_{2}>0$ in order to have
\[
\PP(X_{1}> R_{1}) < \frac{\eps}{2},\quad \PP(X_{2}> R_{2}) < \frac{\eps}{2}.
\] 
For a fixed $\omega \in \Omega$ applying Gronwall Lemma to the function $\phi$ we obtain 
\[
\sup_{t\in[0,T]}\phi(t)(\omega)\leq f(X_{1},X_{2})(\omega).
\]
We now claim that 
\begin{equation*}\label{eq:tightnessmildENDING}
\PP\left(\sup_{t\in[0,T]}\phi(t)> f(R_{1},R_{2})\right)<\eps.
\end{equation*}
In fact we have the following chain of inequalities
\begin{multline*}
\PP\left(\sup_{t\in[0,T]}\phi(t)\leq  f(R_{1},R_{2})\right)\geq\PP\left(\phi(t) \leq R_{1}+R_{2}\int_{0}^{t}\frac{\phi(s)}{(t-s)^{\alpha+1/2}}\,ds\right)\\
\geq \PP \left( (X_{1}\leq R_{1})\cap (X_{2}\leq R_{2}) \right) \geq 1-\PP(X_{1}>R_{1})-\PP(X_{2}>R_{2})\geq 1-\eps.
\end{multline*}
Taking the complement set both sides ends the proof.
\end{proof}
\noindent We are finally able to obtain the following tightness result:
\begin{lem}\label{lemma:tightuN}
The family of laws $\{Q^{N,R,u}\}_{N\in\NN}$ of $\{u^{N,R}\}_{N\in\NN}$, is tight, and hence is relatively compact as a probability measure on $C([0,T]\times\TT^{2})$.
\end{lem}
\begin{proof}
The proof is just an application of Simons embedding in \cite{simon1986compact}. For each $M,R > 0 $ we can consider the following set, for all $\alpha < 1/2$
\begin{multline*}
K_{M,Z} = \bigg\{ u \in C([0,T]\times\TT^{2}) \,|\, \norm{u}_{L^{\infty}([0,T];H^{1+2\alpha}(\TT^{2}))} \leq M, \\ \norm{u}_{W^{1,2}([0,T];H^{-1}(\TT^{2}))} \leq Z \bigg\}.
\end{multline*}
By the Simons Lemma $K_{M,Z}$ is relatively compact in $C([0,T]\times\TT^{2})$. Notice that
\begin{multline*}
Q^{N,R,u}(K^{c}_{M,Z}) =  \PP(u^{N,R}\in K^{c}_{M,Z}) 
\leq \\ \PP\left(\norm{u^{N,R}}_{L^{\infty}([0,T];H^{1+2\alpha}(\TT^{2}))} > M\right)+
\PP\left(   \norm{u^{N,R}}_{W^{1,2}([0,T];H^{-1}(\TT^{2}))} > Z \right)\leq 
\\
\leq \frac{\EE {\norm{u^{N,R}}_{L^{\infty}([0,T];H^{1+2\alpha}(\TT^{2}))}  }}{M} + \eps
\end{multline*}
by lemma \ref{lemma:omegaH^-1}. By Lemma \ref{lemma:tightnessmild} the expected values on the r.h.s. is uniformly bounded with respect to $N$, hence the sequence $\{Q^{N,R,u}\}_{N\in\NN}$ is tight and proof is concluded. 
\end{proof}
\noindent Combining Proposition \ref{prop:tightSN} and Lemma \ref{lemma:tightuN} we obtain the following:
\begin{cor}\label{cor:tightPSR}
The family of laws $\{Q^{N,R}\}_{N\in\NN}$ of the couple $(u^{N,R},S^{N,R})$ is tight, and hence relatively compact as a probability measures on $C([0,T]\times\TT^{2})\times C([0,T];\PP_{1}(\TT^{2}\times\RR^{2}))$.
\end{cor}

\subsection{Convergence of \eqref{eq:PS-VNSR} to \eqref{eq:PDE-VNS}.}\label{subsec:VNSR=VNS}

We will now prove that, under hypothesis on Section \ref{subsec:hypo}, and if $R$ is large enough, then the solution $(u^{R},F^{R})$ of \eqref{eq:PDE-VNSR} coincide with the solution $(u,F)$ of \eqref{eq:PDE-VNS}. To do so we will prove that $u^{R}$ is bounded in $L^{\infty}([0,T]\times\TT^{2})$, independently on $R$. We first summarize all the intermediate result needed for the proof. All the following bounds hold independently on $R$:
\begin{itemize}
\itemsep-0.4em 
\item For all $k \leq 6$ 
\[
\sup_{t\in[0,T]} M_{k}F^{R}_{t} \leq C 
\]
by Lemma \ref{lemma:momentbound} and hypothesis \ref{subsec:hypo};
\item 
\[
 \norm{m_{0}F^{R}}_{L^{\infty}([0,T];L^{2}(\TT^{2})} \leq C,\quad \text{and} \quad  \norm{m_{1}F^{R}}_{L^{\infty}([0,T];L^{2}(\TT^{2})} \leq C
\]
by Lemma \ref{lemma:momentbound} and inequality 1. and 2. in Lemma \ref{lemma:marginalinequ};
\item for all $p>1$
\[
\norm{u^{R}}_{L^{2}([0,T];L^{p}(\TT^{2}))}\leq C_{p}
\]
by Remark \ref{oss:uL2Lp}.
\end{itemize}
We can now formulate a further result, needed in Theorem \ref{teo:u=uR}. 
\begin{lem}\label{lemma:omegaRLinfL2}
There exists  a constant $C$, independent on $R$, such that
\[
\norm{\omega^{R}}_{L^{\infty}([0,T];L^{2}(\TT^{2}))} \leq C.
\]
\end{lem}
\begin{proof}
Computing the time derivative of $\int_{\TT^{2}}\abs{\omega^{R}_{t}}^{2}\,dx$ we obtain
\begin{multline}\label{lemma:omegaRLinfL2:eqmain}
\norm{\omega^{R}_{t}}_{L^{2}(\TT^{2})}^{2} + \int_{0}^{T} \int_{\TT^{2}}\abs{\nabla \omega^{R}_{s}}^{2}\,dx\,ds \lesssim \norm{\omega_{0}}_{L^{2}(\TT^{2})}^{2} +\\
\int_{0}^{t}\int_{\TT^{2}}\omega^{R}_{s}\nabla^{\perp} \cdot \int_{\RR^{2}}(u^{R}_{s}-v)\chi_{R}(u^{R)}F^{R}_{s}\,dv\,dx\,ds.
\end{multline}
Focusing only on the last term of the previous inequality we have
\[
\eqref{lemma:omegaRLinfL2:eqmain} \lesssim \int_{0}^{t} \int_{\TT^{2}} \abs{\nabla \omega^{R}_{s}} \abs{u^{R}_{s}}\int_{\RR^{2}}F^{R}_{s}\,dv\,dx\,ds + \int_{0}^{t} \int_{\TT^{2}} \abs{\nabla \omega^{R}_{s}} \int_{\RR^{2}}\abs{v}F^{R}_{s}\,dv\,dx\,ds
\]
\begin{multline*}
\lesssim \int_{0}^{T}\int_{\TT^{2}}\abs{\nabla\omega^{R}_{s}}^{2}\,dx\,ds + \int_{0}^{T}\int_{\TT^{2}}\abs{u^{R}_{s}}^{2}\left( \int_{\RR^{2}}F^{R}_{s}\,dv \right)^{2}\,dx\,ds\\
+ \int_{0}^{T}\int_{\TT^{2}}\abs{\nabla\omega^{R}_{s}}^{2} \,dx\,ds + \int_{0}^{T}\int_{\TT^{2}}\left( \int_{\RR^{2}}\abs{v}F^{R}_{s}\,dv \right)^{2}\,dx\,ds.
\end{multline*}
Note that
% \begin{multline*}
% \int_{0}^{T}\int_{\TT^{2}}\abs{u^{R}_{s}}^{2}\left( \int_{\RR^{2}}F^{R}_{s}\,dv \right)^{2}\,dx\,ds \leq \int_{0}^{T}\norm{u^{R}_{s}}_{L^{\infty}(\TT^{2})}^{2}\int_{\TT^{2}}\left(\int_{\RR^{2}}F^{R}_{s}\,dv\right)^{2}\,dx\,ds \\
% \leq \sup_{t\in [0,T]}\int_{\TT^{2}}\left(\int_{\RR^{2}}F^{R}_{t}\,dv\right)^{2}\,dx \left(\int_{0}^{T}\norm{u^{R}_{s}}_{L^{\infty}(\TT^{2})}^{2}\,ds\right)\\
% \lesssim \sup_{t\in [0,T]}M_{2}F^{R}_{t} \norm{u^{R}}_{L^{2}([0,T];L^{\infty}(\TT^{2}))}\leq C
% \end{multline*}
\begin{multline*}
\int_{0}^{T}\hspace{-0.1cm}\int_{\TT^{2}}\abs{u^{R}_{s}}^{2}\left( \int_{\RR^{2}}\hspace{-0.1cm}F^{R}_{s}dv \right)^{2}\hspace{-0.1cm}dxds \leq \int_{0}^{T}\hspace{-0.1cm}\norm{u^{R}_{s}}_{L^{4}(\TT^{2})}^{2}\hspace{-0.1cm}\left(\int_{\TT^{2}}\left(\int_{\RR^{2}}F^{R}_{s}\,dv\right)^{4}dx\right)^{\frac{1}{2}}\,\hspace{-0.2cm}ds \\
\leq \sup_{t\in [0,T]}\norm{m_{0}F^{R}_{t}}_{L^{4}(\TT^{2})}^{2} \norm{u^{R}}_{L^{2}([0,T];L^{4}(\TT^{2}))}\\
\lesssim \sup_{t\in [0,T]}(M_{6}F^{R}_{t})^{\frac{1}{2}} \norm{u^{R}}_{L^{2}([0,T];L^{4}(\TT^{2}))}\leq C
\end{multline*}
and 
\begin{equation*}
\int_{0}^{T}\int_{\TT^{2}}\left( \int_{\RR^{2}}\abs{v}F^{R}_{s}\,dv \right)^{2}\,dx\,ds \lesssim_{T}\sup_{t\in[0,T]}M_{4}F^{R}_{t}\leq C
\end{equation*}
still by Lemma \ref{lemma:marginalinequ}.
We conclude by classical Gronwall Lemma.
\end{proof}
\noindent We emphasize that, while it is possible to prove the uniform bound, independently on $R$, state in the previous lemma, it is not possible to obtain the same result directly at the particle level. Namely, we were not able to obtain directly any bound on the vorticity in the particle system \eqref{eq:PS-VNSR}
\[
\EE{\norm{\omega^{N,R}}_{L^{\infty}([0,T];L^{2}(\TT^{2}))}}
\] 
without using the cut off. Having this result would have permitted us to remove the cut off directly at the particle level, without any further complication.

\noindent We finally prove the uniform bound on $u^{R}$:
\begin{prop}\label{prop:uRinfinfuniformbound}
There exists a constant $\mathbf{K}_{u}
$, independent on $R$, such that
\[
\norm{u^{R}}_{\infty}\leq \mathbf{K}_{u}
.
\]
\end{prop}
\begin{proof}
In order to produce the required bound we will bound uniformly the norm of $u^{R}$ in the space $L^{\infty}([0,T];H^{1+2\alpha}(\TT^{2}))$ for any $\alpha < 1/2$. Hence we use the mild formulation for the vorticity equation associated with $u^{R}$:
\[
\partial_t\omega^{R} = \Delta \omega^{R} - u^{R}\cdot \nabla \omega^{R}  - \nabla^{\perp}\cdot\int_{\RR^{2}} (u^{R}-v)\chi_{R}(u^{R)}F^{R}\,dv.
\]
Following the same argument of Lemma \ref{lemma:tightnessmild} we get
\begin{multline*}
\norm{u^{R}_{t}}_{H^{1+2\alpha}(\TT^{2})}\lesssim\norm{\omega^{{R}}_{t}}_{H^{2\alpha}(\TT^{2})}\lesssim  \norm{\omega^{{R}}_{0}}_{H^{2\alpha}(\TT^{2})}\\   
\hspace{-4cm}+\int_{0}^{t} \frac{\norm{u^{{R}}_{s}}_{L^{\infty}(\TT^{2})}\norm{\omega^{{R}}_{s}}_{L^{2}(\TT^{2})}}{(t-s)^{\alpha+1/2}}ds \\
 + \int_{0}^{t}\frac{\norm{u^{R}_{s}}_{L^{\infty}(\TT^{2})}\norm{m_{0}F^{R}_{s}}_{L^{2}(\TT^{2})}}{(t-s)^{\alpha+1/2}} ds + \int_{0}^{t}\frac{\norm{m_{1}F^{{R}}_{s}}_{L^{2}(\TT^{2})}}{(t-s)^{\alpha+1/2}}\,ds
\end{multline*}
\begin{multline*}
\lesssim \norm{\omega^{{R}}_{0}}_{H^{2\alpha}(\TT^{2})} + \norm{\omega^{R}}_{L^{\infty}([0,T];L^{2}(\TT^{2}))}\int_{0}^{t} \frac{\norm{u^{{R}}_{s}}_{H^{1+2\alpha}(\TT^{2})}}{(t-s)^{\alpha+1/2}}ds\\
+ \left(\sup_{t\in [0,T]}M_{2}F^{R}_{t}\right)^{\frac{1}{2}}\int_{0}^{t} \frac{\norm{u^{{R}}_{s}}_{H^{1+2\alpha}(\TT^{2})}}{(t-s)^{\alpha+1/2}}ds  + \left(\sup_{t \in [0,T]}M_{4}F^{R}_{t}\right)^{\frac{1}{2}},
\end{multline*}
by neglecting the cutoff function $\chi_{R}$ which is bounded by one.
By using the uniform bound described at the beginning of Section \ref{subsec:VNSR=VNS}, Lemma \ref{lemma:omegaRLinfL2} and Lemma \ref{lemma:marginalinequ} inequality 1. and 2. we see that all the expression above are bounded independently on $R$ and we conclude by a Gronwall type argument applied to the function $\norm{u^{R}_{t}}_{H^{1+2\alpha}(\TT^{2})}$. 
\end{proof}
Lastly we have the following Theorem:
\begin{teo}\label{teo:u=uR}
If $R \geq \mathbf{K}_{u}+1$, then any weak solution $(u^{R},F^{R})$ of system of PDE \eqref{eq:PDE-VNSR} coincide with the unique bounded weak solutions of system of equations \eqref{eq:PDE-VNS}.
\end{teo}
\begin{proof}
By proposition \ref{prop:uRinfinfuniformbound}, taking $R \geq \mathbf{K}_{u}+1$ we have that the function $\chi_{R}(u^{R}) \equiv 1$, hence system of equation \eqref{eq:PDE-VNSR} reduce to \eqref{eq:PDE-VNS}. Hence, we obtain that the couple $(u^{R},F^{R})$ satisfies system of equation \eqref{eq:PDE-VNS}. By the uniqueness of solution for system of equations \eqref{eq:PDE-VNS}, we obtain $u = u^{R}$ and $F = F^{R}$, ending the proof.
\end{proof}
In order to complete the proof of Proposition \ref{prop:PSRtoVNS} we need only to verify that limit points of the sequence $\{Q^{N,R}\}_{N\in\NN}$ are supported on weak solutions of system of equations \eqref{eq:PDE-VNS}. 
\begin{prop}\label{prop:limitsupportPSVNSRtoPDEVNS}
If $R \geq \mathbf{K}_{u}+1$ limit points of subsequence of $\{Q^{N,R}\}_{N\in\NN}$ are supported on the bounded weak solutions of system of PDE \eqref{eq:PDE-VNS} (see Definition \ref{defi:BWSPDEVNS}).
\end{prop}
\begin{proof}
In order to prove that weak limits are supported on weak solutions, we have to prove that those object satisfies equation \eqref{eq:PDE-VNS} in the weak sense, and that they have the correct regularity. 

\noindent Let us note first that 
the fact that limit points of subsequence have a density on its particle component (corresponding to $S^{N,R}$) which is also in $L^{2}([0,T]\times \TT^{2}\times \RR^{2})$, follows easily from a priori estimates in Corollary \ref{cor:FNL2}. By an analogue argument we have that limits point on the component corresponding to $u^{N,R}$ satisfies the regularity properties on Definition \ref{defi:WSPDEVNSR}, using  Lemma \ref{lemma:omegaL2W12} together with Lemma \ref{lemma:marginalinequ} inequality 3. and 4.

\noindent Moreover by the choice on $R$ we know by the previous theorem that system \eqref{eq:PDE-VNSR} and \eqref{eq:PDE-VNS} are the same.  Consider now $\{Q^{N_{k},R}\}_{k\in\NN}$ a converging subsequence of $\{Q^{N,R}\}_{N\in\NN}$. Recall that $Q^{N_{k},R}$ is a measure on the product space $C([0,T]\times\TT^{2})\times C([0,T];\PP_{1}(\TT^{2}\times\RR^{2}))$, and that is the product measure between $Q^{N_{k},R,S}$, the law fo the empirical measure, and $Q^{N_{k},R,u}$, the law of $u^{N_{k},R}$. Since $Q^{N_{k},R,S}$ is converging weakly on $C([0,T];\PP_{1}(\TT^{2}\times\RR^{2}))$, and $Q^{N_{k},R,u}$ is converging weakly on the space of continuous functions, with respect to the uniform convergence, the fact that those limits are supported on weak solutions of \eqref{eq:PDE-VNS} is classical. 

\noindent In order to complete the proof we need to verify that the density along the particles component is uniformly bounded, as required in Definition \ref{defi:WSPDEVNSR}. However this is verified to be true by applying the maximum principle argued in Section \ref{subsec:maximumprinciple}. Namely, the fact that the limit points along the particles component satisfies system of equations \eqref{eq:PDE-VNS}, where $u$ is uniformly bounded, yields to an uniform bound for the particles density. To do this we only need to verify that, any limit points along the particle component, denote it by $F$, satisfies 
\begin{equation}\label{eq:modDEGOND}
\int_{0}^{T} \int_{\TT^{2}}\int_{\RR^{2}}\abs{v}^{2}F_{s}^{2}\,dx\,dv\,ds<\infty.	
\end{equation}
However, by using Lemma \ref{lemma:FNL4} and interpolation inequality of $L^{p}$ spaces, we have $F \in L^{3}([0,T]\times \TT^{2}\times \RR^{2})$. Also, the uniform bound on the $v$-moments of $F^{N,R}$, provided in Lemma \ref{lemma:M_k}, grants also $M_{4}F$ to be finite. Hence, by an easy computation (see Section \ref{subsec:maximumprinciple}), we see that \eqref{eq:modDEGOND} is satisfied. Thus by the maximum principle
we have $F \in L^{\infty}([0,T]\times \TT^{2}\times \RR^{2})$, hence ending the proof.
\end{proof}

Combining Poroposition \ref{prop:limitsupportPSVNSRtoPDEVNS} with Theorem \ref{teo:uniqueness} we have completed the proof of Proposition \ref{prop:PSRtoVNS}.

%%%%%%%%%%%%%%%%%%%%%%%%%%%%%%%%%%%%%%%%%%%%%%%%%%%%%%%%%
\section{Scaling limit for the full system}\label{sec:limitfull}
The aim of this section is to prove that the cut-off can be removed also in the approximating system $(u^{N,R},S^{N,R})$: the uniform convergence result proved in the previous section, Proposition \ref{prop:PSRtoVNS}, gives a simple but relevant hint to prove the final result of convergence. 
We expect that the converging object $(u^{N,R},S^{N,R})$ inherit the property of boundedness, independently on the parameter $R$, that holds for the limit object. If so, we can remove the cut-off, choosing $R$ large enough from the beginning. The first difficulty in the realization of this intuition is given by the type of convergence which we are dealing with: convergence in law. We will overcome this technicality, appealing to the Skorohod's Theorem to strengthen the convergence. 

We will first state and prove a general result for almost sure convergence of random variables. Later on, in order to utilize such criterion, we will make us Skorohod's Theorem and we will understand our particle systems in a \emph{path-by-path} sense: we will give a precise definition of path-by-path solutions and prove a uniqueness result for such kind of solutions. The application of the above mentioned criterion to our case will let us transfer the property of convergence from the sequence $(u^{N,R},S^{N,R})$ to $(u^{N},S^{N})$. 

In the rest of the section we will always assume to have taken 
\[
R = \max(\mathbf{K}_{u}+1,\norm{u}_{L^{\infty}([0,T]\times \TT^{2})}+1)
\]
where the constant $\mathbf{K_{u}}$ has been defined in Proposition \ref{prop:uRinfinfuniformbound}. This choice will assure that Proposition \ref{prop:PSRtoVNS} is verified. The condition that $R$ is greater than $\norm{u}_{L^{\infty}([0,T]\times \TT^{2})}+1$ is needed in order to let the sequence of $u^{N,R}$ to inherit the uniform boundedness of the limit $u$. This process will be clarified later.

\subsection{Convergence criterion}
We now present the general criterion that we will use to obtain the convergence of the sequence $(u^{N},S^{N})_{N\in\NN}$ from that of $(u^{N,R},S^{N,R})_{N\in\NN}$.
The framework of this criterion is pretty general. We preferred to isolate it an state it in its general form, rather than in our specific case, in order to make the underlying idea more evident. 

\begin{teo}[General Principle]\label{teo:GP}
Let $(\Omega, \mathcal{F},\PP)$ a probability space and let $(E,d_{E})$ a separable metric space. Let $\{X_{N}\}_{N\in\NN}$ and $\{Y_{N}\}_{N\in\NN}$ two sequences of random variables taking values in $E$ and let $x$  be a point in $E$. Moreover, suppose that for each $N\in \NN$, there exist two collections of subset $S^{X}_{N}(\omega)\subseteq E$ and $S^{Y}_{N}(\omega)\subseteq E$, indexed by $\omega \in \Omega$.
Assume further that the following conditions are satisfied:
\begin{enumerate}
\item 
\[
X_{N} \xrightarrow{N\rightarrow\infty} x \in E \quad \PP\text{-a.s.};
\]
\item denoting 
\[
\Omega_{S} = \left\{ \omega \in \Omega \,|\, \sharp S^{Y}_{N}(\omega)\leq 1 \quad \forall N \in \NN \right\}
\]
where by $\sharp A$ we mean the cardinality of the set $A$, we have
\[
\PP(\Omega_{S}) = 1;
\]
\item denoting 
\[
\Omega_{X} = \left\{\omega \in \Omega \,|\, X_{N}(\omega) \in S^{X}_{N}(\omega)\;\forall N \in \NN \right\},
\]
\[
\Omega_{Y} = \left\{\omega \in \Omega \,|\, Y_{N}(\omega) \in S^{Y}_{N}(\omega)\; \forall N \in \NN\right\},
\]
we have
\[
\PP(\Omega_{X}) = \PP(\Omega_{Y}) = 1;
\]
\item 
\[
B_{E}(x,1) \cap S^{X}_{N}(\omega) \subseteq S^{Y}_{N}(\omega)\quad \forall N \in \NN,\,\forall \omega \in \Omega.
\]
Then the sequence $\{Y_{N}\}_{N\in\NN}$ converges in $E$ to the same limit of the sequence $\{X_{N}\}_{N\in\NN}$
\[
Y_{N} \xrightarrow{N\rightarrow\infty} x \in E \quad \PP\text{-a.s.}
\]
\end{enumerate}
\end{teo}
\begin{proof}
Consider the set 
\[
\Omega_{C,X} := \left\{\omega\in\Omega \,|\, d(X^{N}(\omega),x)_{E} \stackrel{N}{\to} 0\right\}
\]
and 
\[
\Omega_{C,Y} := \left\{\omega\in\Omega \,|\, d(Y^{N}(\omega),x)_{E} \stackrel{N}{\to} 0\right\}
\]
Note that, by property 1. the set $\Omega_{C,X}$ has full measure $\PP(\Omega_{C,X})=1$. \\
We will prove that 
\begin{equation}\label{eq:GPfresProof}
\Omega_{S}\cap \Omega_{C,X}\cap \Omega_{X}\cap \Omega_{Y} \subseteq \Omega_{C,Y}
\end{equation}
thus implying the thesis being $\PP(\Omega_{S})=\PP(\Omega_{X}) = \PP(\Omega_{Y}) = 1$ by property 2. and 3. 
To do so let us consider the set
\[
\Omega_{1} = \left\{ \omega \in \Omega \,|\, \exists N(\omega) \, d(X_{N}(\omega),x)\leq 1 \,\forall N > N(\omega) \right\}
\]
and note that 
\[
\Omega_{X,C} \subseteq \Omega_{1}.
\]
% In fact, by the convergence of $X^{N}(\omega)$  to $x$, one has immediately the convergence of the norms of those object 
% \[
% \norm{X_{N}(\omega)}_{E}\to \norm{x}_{E}
% \]
% and hence the bound for $N$ large enough (depending on $\omega$). 
Now define
\[
\Omega_{2} = \left\{\omega\in\Omega \,|\, X_{N}(\omega) = Y_{N}(\omega)\,\forall N > N(\omega)\right\}
\]
where $N(\omega)$ is defined for each $\omega$, in the set $\Omega_{1}$. We claim that
\begin{equation}\label{eq:GPiresproof}
\Omega_{S}\cap\Omega_{X,C}\cap\Omega_{X}\cap\Omega_{Y} \subseteq \Omega_{2}.
\end{equation}
Take $\omega \in \Omega_{S}\cap\Omega_{X,C}\cap\Omega_{X}\cap\Omega_{Y}$. Hence if $N > N(\omega)$, given that $\omega$ lies in $\Omega_{X,C}$, it also lies in $\Omega_{1}$, thus we have $X_{N}(\omega) \in B_{E}(x,1)_{E}$. Moreover, $\omega$ lies also in $\Omega_{X}$, hence $X_{N}(\omega) \in S^{X}_{N}(\omega)$. By property 4. we conclude $X_{N}(\omega) \in S^{Y}_{N}(\omega)$.
Furthermore $\omega \in \Omega_{Y}$ implies $Y_{N}(\omega) \in S^{Y}_{N}(\omega)$, but $\omega$ is also in $\Omega_{S}$ hence by property 2. $S^{Y}_{N}(\omega)$ is a singleton, hence $S^{Y}_{N}(\omega) = \{Y_{N}(\omega)\}$. Since $X_{N}(\omega) \in S^{X}_{N}(\omega)$ and $S^{Y}_{N}(\omega) = \{Y_{N}(\omega)\}$ we obtain $X_{N}(\omega)=Y_{N}(\omega)$ and we have proven condition \eqref{eq:GPiresproof}.\\
Finally, we can prove condition \eqref{eq:GPfresProof}: 
taking $\omega \in \Omega_{S}\cap\Omega_{X,C}\cap\Omega_{X}\cap\Omega_{Y}$, we have that $\forall \eps>0$ there exists $N_{\eps}(\omega)$, such that 
\[
d(X_{N}(\omega),x)_{E} < \eps\quad \forall N > N_{\eps}(\omega)
\]
By condition \eqref{eq:GPiresproof} $\omega$ lies also in $\Omega_{2}$, hence
\[
X_{N}(\omega) = Y_{N}(\omega)\quad \forall N > N(\omega).
\]
Calling $\overline{N}(\omega) = \max(N_{\eps}(\omega),N(\omega))$ we conclude 
\[
d(Y_{N}(\omega),x)_{E} < \eps\quad \forall N > \overline{N}_{\eps}(\omega)
\]
and hence $\omega \in \Omega_{Y,C}$. Thus the proof is concluded.

\end{proof}

% To apply this strategy we need first to introduce some notation and tools.
% Observe first that both in \eqref{eq:PS-VNS} and  \eqref{eq:PS-VNSR} the equation for the particles are only affected by additive noise. Hence, we can consider a path-by-path formulation for the particle system, which consist in considering the SDEs for ${\omega} \in {\Omega}$ fixed: what we obtain are classical ODEs where the Brownian motion plays the role of a given driving term. With this interpretation of the equations, in order to have existence of solutions for $\omega$ fixed, we need to impose that the driving Brownian motions $(B^{i}_{t}(\omega))_{i \in \NN}$ are continuous. To shorten the notation introduce $\mathcal{B} := (B^{i}_{\cdot})_{i \in \NN}$ and define
% \[
% \Omega_{\mathcal{B}} := \left\{ \omega \in \Omega \,|\, B_{t}(\omega) \text{ is continuous } \forall t \in [0,T]\,\forall B \in \mathcal{B} \right\}
% \]
% (note that $\PP(\Omega_{\mathcal{B}}) = 1$ given that $\mathcal{B}$ is made of a countable sequence of Brownian motions). Thus on $\Omega_{\mathcal{B}}$ there exist path-by-path solutions both to system of equations \eqref{eq:PS-VNS} and \eqref{eq:PS-VNSR}: in fact for both systems the interaction with the fluid is smoothed by a convolution with a mollifier, hence is a regular function. We can then apply classical result for ODEs with continuous coefficient obtaining existence. 
\subsection{Path by Path solutions for \eqref{eq:PS-VNS}}
We will now focus on the problem of uniqueness for \emph{path-by-path} solutions. The issue of uniqueness for this class of solutions is very difficult: very few result are know before the work of \cite{davie2007}. 
The analysis of such kind of problem for \eqref{eq:PS-VNS} will be a key point of the proof of Theorem \ref{teo:PStoVNS}. In fact, to apply Theorem \ref{teo:GP} to our case, we will see that strong uniqueness in the sense of SDEs, which is more classical than that path-by-path, will not be enough. We now recall the concept of path-by-path solutions and uniqueness in this class. We will discuss this topic in the specific case that is needed here, the system of PDE-SDEs \eqref{eq:PS-VNS}. 

Recall system of equation \eqref{eq:PS-VNS} and note that, in the equation for the particle position and velocity $(X^{i,N}_{t},V^{i,N}_{t})$ the noise is pure additive Brownian motion, i.e. the diffusion coefficient is constant. For this reason no It\^o integral is involved into the equations and one can understand system of equations \eqref{eq:PS-VNS} in its integral form as a coupling PDE-ODEs, where the Brownian motions plays the role of a given external force. This perspective is outlined in the following system
\begin{equation}\label{eq:PS-VNS-pathbypath}
\begin{cases}
\partial_{t}u^{N}= \Delta u^{N}- u^{N}\cdot \nabla u^{N} -\nabla \pi^{N} -\frac{1}{N}\sum_{i=1}^{N} (u^{N}_{\eps_{N}}(X^{i,N}_{t})-V^{i,N}_{t})\delta^{\eps_{N}}_{X^{i,N}_{t}}\\
\div(u^{N})=0,\\
\begin{cases}
X^{i,N}_{t}=X_{0}^{i}+\int_{0}^{t}V^{i,N}_{s}\,ds\vspace{0.2cm}\\
V^{i,N}_{t}=V_{0}^{i}+\int_{0}^{t}(u^{N}_{\eps_{N}}(X^{i,N}_{s})-V^{i,N}_{s})\,ds+\sigma B^{i}_{t}(\omega)
\end{cases}i=1,\dots,N
\end{cases}	
\end{equation}
where $B^{i}_{t}(\omega)$ stands for a single realization of a Brownian path for fixed $\omega \in \Omega$. We now introduce the set of path-by-path solutions for a given realization of $\omega \in \Omega$ and for fixed $N \in \NN$:
\begin{multline}\label{eq:setofsolutionPBP}
S_{N}({\omega}) =  \bigg\{ \bigg(w,\big(x^{i}_{\cdot},v^{i}_{\cdot}\big)_{i=1,\dots,N}\bigg)\in C([0,T]\times \TT^{2})\times C([0,T];\TT^{2}\times \RR^{2})^{N}\,\text{s.t.}\\ \bigg(w,\big(x^{i}_{\cdot},v^{i}_{\cdot}\big)_{i=1,\dots,N}\bigg) \text{ solves } \eqref{eq:PS-VNS-pathbypath} \text{ with additive noise }(B^{i}_{t}(\omega))_{i=1,\dots,N} \bigg\}.
\end{multline}
Roughly speaking $S_{N}(\omega)$ is the set of curves that solves \eqref{eq:PS-VNS-pathbypath} in a deterministic setting for a prescribed realization of a Brownian path (identified by $\omega$). We do not give a precise definition of existence of path-by-path solutions. We remark that existence of weak or strong solutions in an SDE settings imply that the set $S_{N}(\omega)$ is non empty with probability one. 
We now focus our attention to the topic of uniqueness. 
\begin{defi}[Uniqueness of path-by-path solutions]\label{defi:uniquenessPBP}
Given a natural number $N$ we say that there is path-by-path uniqueness for system of equations \eqref{eq:PS-VNS} with $N$ particles, if there exist a set ${\Omega}_{S}\subseteq \Omega$ with probability one $\PP(\Omega_{S}) = 1$ such that 
\[
\sharp S_{N}(\omega) \leq 1 \quad \forall \omega \in \Omega_{S}
\]
where $\sharp A$ stands for the cardinality of the set $A$.
\end{defi}
Opposite to the case of existence, uniqueness of path-by-path solutions is a much more difficult topic: 
in fact uniqueness in this class is a stronger notion that weak or strong uniqueness for SDE. In Definition \ref{defi:uniquenessPBP} no measurability with respect to the probability space $({\Omega}, {\mathcal{F}},\{{\mathcal{F}}_{t}\},{\PP})$ is required. In case of uniqueness for SDE a much more richer structure is available, given that solutions are required at least to be adapted to the filtration $\mathcal{F}_{t}$.

We now prove a path-by-path uniqueness result for system of equation \eqref{eq:PS-VNS}. Some result about path-by-path uniqueness for SDEs are already known: Davie in \cite{davie2007} prove the result for a single SDE with pure additive Brownian noise and only bounded measurable drift. This type of result for low regularity drift functions, less than locally Lipschitz, are very difficult. 
In our case, the drift appearing into the particle equations $(X^{i,N}_{t},V^{i,N}_{t})$ is even more regular than Lipschitz: in fact the function $u^{N}_{\eps_{N}}(t,x)$ is $C^{\infty}$ in the space variable due to the convolution with the $C^{\infty}$ function $\theta^{\eps_{N}}(x)$. However, the case here is slightly different from the case of a single SDE due to the strong coupling with the Navier-Stokes equation that introduce additional difficulty.

\begin{prop}\label{prop:pathuniqueness}
On the probability space $(\Omega, \mathcal{F},\left\{\mathcal{F}_{t}\right\},\PP)$ consider 
\[
\Omega_{\mathcal{B}} = \left\{ \omega \in \Omega \,|\, B^{i}_{t}(\omega) \text{ is continuous on }[0,T] \,\forall i\in\NN \right\}\subseteq\Omega
\]
the set where all the Brownian motion $(B^{i})_{i\in\NN}$ are continuous, which is of full measure with respect to $\PP$. Then, for all $N \in \NN$ we have uniqueness path-by-path for system of equation \eqref{eq:PS-VNS} with $N$ particles, namely
\[
\sharp S_{N}(\omega)\leq 1 \quad \forall \omega \in \Omega_{\mathcal{B}}.
\]		
\end{prop}
\begin{proof}
For a matter of simplicity we prove the result in the case $N=1$: the generalization to the case where $N$ is arbitrary is straightforward.
Also, to make the notation less heavy, we will omit the dependence on $N$ and ${\omega}$ indicating with $u_{t}$ the variable $u^{N}_{t}({\omega})$ and with $(X_{t},V_{t})$ the couple of variables $(X^{1,1}_{t},V^{1,1}_{t})({\omega})$. Also the mollifier $\theta^{0,\eps_{N}}$ will be labeled simply $\theta$, again for a matter of clarity. In our simplification, the system becomes:	
\[
\begin{cases}
\partial_{t} u = \Delta u - u \cdot \nabla u -\nabla \pi - \left( (\theta*u)(X_{t})-V_{t}\right)\theta(x-X_{t})\\
\div(u) = 0\\
\begin{cases}
\dot{X}_{t} &= V_{t}\\
\dot{V}_{t} &= \left( (\theta*u_{t})(X_{t})-V_{t} \right)+B_{t}.
\end{cases}
\end{cases}
\]
% We claim that the path-by-path solutions are unique in the subset of $\widetilde{\Omega}$ where all the Brownian motions involved into the equations are continuous.\\
% For the case $N=1$ it is clear that the single Brownian motion $B_{t}$ is continuous with probability one. 
% For the general case ($N>1$) we previously introduced the set $\mathbf{B} = (\widetilde{B}^{i,N}_{\cdot})_{N\in\NN,i\leq N}$. We recall that the set $\widetilde{\Omega}_{\mathbf{B}}$, i.e. the subset of $\tilde{\Omega}$ where all the Brownian motions in $\mathbf{B}$ are continuous, have probability one. Thus, if $\widetilde{\omega}\in \widetilde{\Omega}_{\mathbf{B}}$ then each Brownian motion involved in $\eqref{eq:PS-VNS}$ is a continuous function. Hence, we will restrict ourself to the set of continuous trajectory and prove uniqueness.\\
\noindent We will now consider two solutions $(u,X,V)$ and $(u',X',V')$, with $(u_{0},X_{0},V_{0})=(u'_{0},X'_{0},V'_{0})$, and we will apply Gronwall Lemma to the quantity
\[
\abs{X_{t}-X'_{t}}+\abs{V_{t}-V'_{t}}+\norm{u_{t}-u'_{t}}_{H^{1+2\alpha}(\TT^{2})},
\]
for $\alpha < \frac{1}{2}$.\\
We start by computing the distance of velocities, recalling that $V_{0} = V'_{0}$ and $B_{t}$ is the same given function for the two solutions
\begin{multline*}\abs{V_{t}-V'_{t}} \leq \int_{0}^{t} \left[(\theta*u_{s})(X_{s})-(\theta*u_{s}')(X'_{s})\right]\,ds + \int_{0}^{t}\abs{V_{s}-V'_{s}}\,ds\\
\leq \int_{0}^{t} \left[(\theta*u_{s})(X_{s})-(\theta*u_{s}')(X_{s})\right]\,ds+\int_{0}^{t}  \left[(\theta*u_{s}')(X_{s})-(\theta*u_{s}')(X'_{s})\right]\,ds  \\
\hspace{-6cm}+ \int_{0}^{t}\abs{V_{s}-V'_{s}}\,ds\\
\lesssim \int_{0}^{t}\norm{u_s-u'_s}_{H^{1+2\alpha}(\TT^{2})} ds +\int_{0}^{t} \abs{X_{s}-X'_{s}}\,ds + \int_{0}^{t}\abs{V_{s}-V'_{s}}\,ds 
\end{multline*}
where we have used both the Lipschitzianity of $\theta * u_{s}$ and its boundedness in $L^{\infty}(\TT^{2})$, and the embedding $H^{1+2\alpha}(\TT^{2}) \hookrightarrow C(\TT^{2})$.\\
For the $X$ component we simply have
\begin{equation*}
\abs{X_t-X'_t}\leq \int_{0}^{t} \abs{V_s-V'_s} ds.
\end{equation*}
The main difficulty consists in estimating $\norm{u_{t}-u'_{t}}_{H^{1+2\alpha}(\TT^{2})}$. As done in previous sections we approach the problem through the vorticity formulation. Call $\omega$ and $\omega'$ the vorticity associated with $u$ and $u'$. As in Lemma \ref{lemma:tightnessmild}, by the mild formulation of $\omega - \omega'$  we have
\begin{align}
\label{eq:pathuniq-mild1}
\norm{\omega_{t}-\omega_{t}'}_{H^{2\alpha}(\TT^{2})} &\leq \int_{0}^{t} \norm{(I-\Delta)^{\alpha} e^{(t-s)\Delta}u_s\cdot \nabla (\omega_{s}-\omega_{s}')}_{L^{2}(\TT^{2})} ds\\
\label{eq:pathuniq-mild2}
&+\int_{0}^{t} \norm{(I-\Delta)^{\alpha} e^{(t-s)\Delta}({u}_s-u_{s}')\cdot \nabla \omega'_s }_{L^{2}(\TT^{2})} ds\\
\label{eq:pathuniq-mild3}
&+\int_{0}^{t} \bigg|\bigg|(I-\Delta)^{\alpha} e^{(t-s)\Delta}\nabla^\perp\cdot \Lambda_{u,X,V}(s)\bigg|\bigg|_{L^{2}(\TT^{2})}\, ds
\end{align} 
where 
\[
\Lambda_{u,X,V}(s):=\bigg[\left( (\theta*u)(X_{s})-V_{s}\right)\theta(x-X_{s}) - \left( (\theta*u')(X'_{s})-V'_{s}\right)\theta(x-X'_{s})\bigg].
\]
We now deal with each of the terms above separately. We strictly follow the same computation of Lemma \ref{lemma:tightnessmild}, staring from \eqref{eq:pathuniq-mild1}:
\begin{multline*}
\eqref{eq:pathuniq-mild1} \lesssim \int_{0}^{t} \frac{\norm{u_{s}}_{C(\TT^{2})}\norm{\omega_{s}-\omega_{s}'}_{L^{2}(\TT^{2})}}{\abs{t-s}^{\alpha+1/2}}  ds \lesssim \\ \lesssim \norm{u}_{\infty}\int_{0}^{t} \frac{\norm{u_{s}-u_{s}'}_{H^{1+2\alpha}(\TT^{2})}}{\abs{t-s}^{\alpha+1/2}}  ds.
\end{multline*}
\begin{multline*}
\eqref{eq:pathuniq-mild2}\lesssim \int_{0}^{t} \frac{\norm{u_{s}-u_{s}'}_{C(\TT^{2})}\norm{\omega_{s}'}_{L^{2}(\TT^{2})}}{\abs{t-s}^{\alpha+1/2}}  ds\lesssim 
\\ 
\lesssim \norm{\omega'}_{L^{\infty}([0,T];L^{2}(\TT^{2}))}\int_{0}^{t}\frac{\norm{u_{s}-u_{s}'}_{H^{1+2\alpha}(\TT^{2})}}{\abs{t-s}^{\alpha+1/2}}\,ds.
\end{multline*}
In the same way we have
\[
\eqref{eq:pathuniq-mild3} \lesssim \int_{0}^{t}\frac{\norm{\Lambda_{u,X,V}(s)}_{L^{2}(\TT^{2})}}{\abs{t-s}^{\alpha+1/2}}\,ds
\]
We proceed now by adding and subtracting the right quantities from $\Lambda_{u,X,V}(s)$ obtaining
\begin{equation*}
\abs{\Big[ (\theta*u_{s})(X_{s})-V_{s}\Big]\theta(x-X_{s})-
\Big[ (\theta*u_{s}')(X'_{s})-V'_{s}\Big]\theta(x-X'_{s})}\leq
\end{equation*}
\begin{align*}
&\leq \theta(x-X_{s})\abs{(\theta*u_{s})(X_{s})-(\theta*u_{s}')(X_{s})}\\
&+\theta(x-X_{s})\abs{(\theta*u_{s}')(X_{s})-(\theta*u_{s}')(X_{s}')}\\
&+ u_{s}'(X_{s}')\abs{\theta(x-X_{s})-\theta(x-X_{s}')}\\
&+ \theta(x-X_{s})\abs{V_{s}-V_{s}'}\\
&+\abs{V_{s}}\abs{\theta(x-X_{s})-\theta(x-X_{s}')}
\end{align*}
\begin{align*}
\lesssim \norm{u_{s}-u_{s}'}_{H^{1+2\alpha}(\TT^{2})} +\abs{X_{s}-X_{s}'} +\abs{V_{s}-V_{s}'}
\end{align*}
by using the boundedness of $u$ and $u'$, the Lipschitzianity of $(\theta*u)$, the boundedness of $\abs{V_{s}}$ and that of $\theta$. Hence we obtained
\[
\eqref{eq:pathuniq-mild3} \lesssim \int_{0}^{t}\frac{\norm{u_{s}-u_{s}'}_{H^{1+2\alpha}(\TT^{2})} +\abs{X_{s}-X_{s}'} +\abs{V_{s}-V_{s}'}}{\abs{t-s}^{1/2+\alpha}}\,ds
\]
We conclude by a standard Gronwall type inequality.
\end{proof}

\subsection{Proof of Theorem \ref{teo:PStoVNS}}
We finally have all the ingredients to prove Theorem \ref{teo:PStoVNS}. Since the proof is quite technical we outline the general strategy first. \\
From Proposition \ref{prop:PSRtoVNS} we have obtained convergence in Law of the sequence $({u}^{N,R},{S}^{N,R})$ to the unique weak solution of \eqref{eq:PDE-VNS}, call it $(u,F)$. We aim to obtain the same result for the sequence $({u}^{N},{S}^{N})$, namely to prove Theorem \ref{teo:PStoVNS}. To do so, we will apply the convergence criterion stated in Theorem \ref{teo:GP}, to transfer the convergence from one sequence to another. However, Theorem \ref{teo:GP} requires the sequences of random variables involved, to converge almost surely in the appropriate topology, while Proposition \ref{prop:PSRtoVNS} grants us only convergence in law. Hence, to overcome this problem, we will first appeal 
to a slight variation of Skrohood representation Theorem, Lemma \ref{lemma:skorvariation}, applied to the sequence $({u}^{N,R},{S}^{N,R})_{N\in\NN}$ in order to obtain almost sure convergence from convergence in law. Let us omit some technicalities concerning Skrohood Theorem, whose detail will be clarified later, and assume now that the sequence $({u}^{N,R},{S}^{N,R})$ is converging almost surely to $(u,F)$. We will apply Theorem \ref{teo:GP} by taking 
\[
X_{N} = (u^{N,R},S^{N,R}),\quad Y_{N} = (u^{N},S^{N}), \quad x = (u,F).
\]
Still avoiding some technicalities we will chose
\[
S_{N}^{X}(\omega) = \text{the set of path-by-path solutions of \eqref{eq:PS-VNSR}} 
\]
and
\[
S_{N}^{Y}(\omega) = \text{the set of path-by-path solutions of \eqref{eq:PS-VNS}} .
\]
With this choice we will see that conditions [1-4] stated in Theorem \ref{teo:GP} will be satisfied. We can now outline the reasoning behind the hypothesis of Theorem \ref{teo:GP} in the following scheme: 
\begin{enumerate}
\item correspond to Proposition \ref{prop:PSRtoVNS}, that is the convergence of $(u^{N,R},S^{N,R})$ to the limit point $(u,F)$;
\item resemble to the path-by-path uniqueness result, Proposition \ref{prop:pathuniqueness};
\item state that 
$(u^{N,R},S^{N,R})$ is a path-by-path solution of \eqref{eq:PS-VNSR} and the analogue for $(u^{N},S^{N})$;
\item express the fact that path-by-path solutions of \eqref{eq:PS-VNSR} which are also bounded from above, also satisfies \eqref{eq:PS-VNS} if the parameter $R$ is large enough. This is the same idea used to prove Theorem \ref{teo:u=uR} when we proved that two PDE system coincide for large $R$. 
\end{enumerate}

We now remark the importance of having path-by-path uniqueness. Imagine to replace condition 2. in Theorem \ref{teo:GP}, with some condition that mimics the idea of strong uniqueness for SDE, instead of that for path-by-path uniqueness. A possible modification is the following:\\
\emph{Condition 2.bis}: For all $N \in \NN$ and for every $Z$ $E$-valued random variable, if
\[
\PP(Z(\omega)\in S^{Y}_{N}(\omega)) = 1 
\]
then
\[
\PP(Z(\omega) = Y_{N}(\omega)) = 1.
\]
Now, following the proof of Theorem \ref{teo:GP}, we can proceed into the proof up to a certain point. Specifically we can prove that the set
\[
\left\{X_{N}(\omega) \in S^{Y}_{N}(\omega)\quad \forall N > N(\omega) \right\} 
\]
is of full measure with respect to $\PP$. However, there is no way to apply condition \emph{2.bis}, to conclude that 
\[
\PP(X_{N}(\omega) = Y_{N}(\omega)\quad \forall N>N(\omega)) = 1
\]
as it would be needed to end the proof.\\
We now recall and prove a small variation of Skorohod's Theorem, that we will need in the proof of Theorem \ref{teo:PStoVNS}.
\begin{lem}\label{lemma:skorvariation}
Let $(\Omega,\mathcal{F},\PP)$ be a probability space and let $(X_{N},Y_{N})_{N \in \NN}$ be a sequence of  random variables defined on $\Omega$, taking values on a separable metric space $E\times F$. Assume that $F$ is also a Banach space and $Y_{N} \in L^{1}(\Omega;F)$ for each $N\in\NN$. Let also $X : \Omega \to E$ be a random variable and assume further that 
\[
X_{N}\stackrel{\mathcal{L}aw}{\to}X.
\] 
Then, there exist a probability space $(\widetilde{\Omega},\widetilde{\mathcal{F}},\widetilde{\PP})$ and random variables defined on the new probability space $(\widetilde{X}_{N},\widetilde{Y}_{N})_{N\in\NN}$, $\widetilde{X}$ such that
\[
(\widetilde{X}_{N},\widetilde{Y}_{N}) \stackrel{\mathcal{L}aw}{\sim} (X_{N},Y_{N}), \quad \widetilde{X}\stackrel{\mathcal{L}aw}{\sim}X
\]
and
\[
\widetilde{X}_{N}{\to}\widetilde{X}\quad \widetilde{\PP}\text{-almost-surely.}
\]
\end{lem}
\begin{proof}
The proof will rely on the classical Skorohod's Theorem, 
see \cite{billingsley2013convergence}.\\
Call $c_{N} := \EE{\norm{Y_{N}}_{F}}$, and introduce $a_{N}  = Nc_{N}$. Consider now the sequence $Z_{N} := Y_{N}/a_{N}$ and notice that 
\[
Z_{N}\stackrel{\mathcal{L}aw}{\to}0
\]
since the convergence also holds in probability. Now, applying Skorohod's Theorem to the sequence $(X_{N},Z_{N})$ we obtain that there exist a new probability space $(\widetilde{\Omega},\widetilde{\mathcal{F}},\widetilde{\PP})$ and random variables $(\widetilde{X}_{N},\widetilde{Z}_{N})_{N\in\NN}$, $\widetilde{X}$ such that
\[
(\widetilde{X}_{N},\widetilde{Z}_{N}) \stackrel{\mathcal{L}aw}{\sim} (X_{N},Z_{N}), \quad \widetilde{X}\stackrel{\mathcal{L}aw}{\sim}X
\]
and 
\[
\widetilde{X}_{N}{\to}\widetilde{X}\quad \widetilde{\PP}\text{-almost-surely.}
\]
Introduce $\widetilde{Y}_{N}:= a_{N}\widetilde{Z}_{N}$ and observe that $(\widetilde{X}_{N},\widetilde{Y}_{N}) \stackrel{\mathcal{L}aw}{\sim} (X_{N},Y_{N})$. Hence the proof concludes.
\end{proof} 
\begin{proof}[Proof of Theorem \ref{teo:PStoVNS}]
As explained the above the proof is divided in three steps: first we apply a variation Lemma of \ref{lemma:skorvariation} to the sequence $(u^{N,R},S^{N,R})$ to obtain almost sure convergence on a new probability space. Second, we will see that the new random variables obtained, on the new probability space satisfies the same equations as the original one. Lastly, we apply the general principle Theorem \ref{teo:GP} to transfer the convergence from $(u^{N,R},S^{N,R})$ to $(u^{N},S^{N})$.

\noindent\emph{Step 1:} Let us first introduce some notation. For each $N \in \NN $ we first introduce $\mathbf{X}^{N,R}\in C([0,T];\TT^{2})^{\NN} $  and $\mathbf{V}^{N,R}\in C([0,T]; \RR^{2})^{\NN} $  defined as
\[
\mathbf{X}^{N,R}(i) = 
\begin{cases}
&X^{i,N,R}_{\cdot}	\quad \text{ if $i \leq N$, }\\
&{0}\quad \text{ otherwise,}
\end{cases}\quad 
\mathbf{V}^{N,R}(i) = 
\begin{cases}
&V^{i,N,R}_{\cdot}	\quad \text{ if $i \leq N$, }\\
&{0}\quad \text{ otherwise.}
\end{cases}
\]
where $0$ stands for the function which is identically zero. Roughly speaking $\mathbf{X}^{N,R}$ (respectively $\mathbf{V}^{N,R}$) is the sequence of functions, where the first $N$ elements are the particles trajectories $X^{i,N,R}_{\cdot}$, and the others all identically zero. 
Now we apply Lemma \ref{lemma:skorvariation} to the sequence 
\[
(u^{N,R},S^{N,R},(B^{i})_{i\in\NN},\mathbf{X}^{N,R},\mathbf{V}^{N,R})_{N\in\NN}
\]
where 
\[
(u^{N,R},S^{N,R},(B^{i})_{i\in\NN}) \xrightarrow{\mathcal{L}aw} (u,F,(B^{i})_{i\in\NN})
\]
and we just need to verify that $(\mathbf{X}^{N,R},\mathbf{V}^{N,R})$ is integrable with respect to $\PP$ for each $N \in \NN$. However this is true because
\begin{multline*}
\EE{\norm{\mathbf{X}^{N,R}}_{L^{\infty}([0,T];\TT^{2})^{\NN}}} = \EE{\max_{i\leq N}\sup_{t\in[0,T]} \norm{X^{i,N,R}_{t}}} \\ \leq N \EE{\sup_{t\in[0,T]} \norm{X^{i,N,R}_{t}}}\leq C\cdot N 
\end{multline*}
by using exchangeability and by the fact that $\EE{\sup_{t\in[0,T]} \norm{X^{i,N,R}_{t}}} \leq C$
due to the presence of the cutoff in system of equations \ref{eq:PS-VNSR}. The same result holds for every $\mathbf{V}^{N,R}$ by using the same argument.

\noindent We can now apply Lemma \ref{lemma:skorvariation}. Hence there exists a new filtered probability space  $(\widetilde{\Omega}, \widetilde{\mathcal{F}},\{\widetilde{\mathcal{F}}_{t}\},\widetilde{\PP})$ and new sequences of random variables 
\[
(\widetilde{u}^{N,R},\widetilde{S}^{N,R},(\widetilde{B}^{i,N})_{i\leq N},\widetilde{\mathbf{X}}^{N,R},\widetilde{\mathbf{V}}^{N,R})_{N\in\NN}
\]
that shares the same laws of the initial sequences
\[
(\widetilde{u}^{N,R},\widetilde{S}^{N,R},(\widetilde{B}^{i,N})_{i\leq N},\widetilde{\mathbf{X}}^{N,R},\widetilde{\mathbf{V}}^{N,R}) \stackrel{\mathcal{L}aw}{\sim}(u^{N,R},S^{N,R},(B^{i})_{i\leq N},\mathbf{X}^{N,R},\mathbf{V}^{N,R})
\]
for all $N \in \NN$, and that satisfies
\[
\left(\widetilde{u}^{N,R},\widetilde{S}^{N,R}\right)\xrightarrow{N\rightarrow\infty} \left(u,F\right)\quad \widetilde{\PP}\text{-a.s.}
\]

\noindent \emph{Step 2:} We now verify that the new random variables satisfies the same equations as the original ones, namely system of equations \eqref{eq:PS-VNSR}. Moreover, in order to apply Theorem \ref{teo:GP} we also need to have on the new probability space an analogue of $(u^{N},S^{N})$, that still satisfies system of equations \eqref{eq:PS-VNS} and of which we will prove the convergence. 
% In fact, the new random variables on the new probability space need to satisfy the same system of equation of the original random variables. Moreover, we need also to transport to $(\widetilde{\Omega}, \widetilde{\mathcal{F}},\{\widetilde{\mathcal{F}}_{t}\},\widetilde{\PP})$
% also the particle system and the empirical measure without the cut off: 
Namely:
\begin{enumerate}
\item Denoting by $\widetilde{X}^{i,N,R}$ and $\widetilde{V}^{i,N,R}$ the first $N$ components of $(\widetilde{\mathbf{X}}^{N,R},\widetilde{\mathbf{V}}^{N,R})$, corresponding to those that are non zero, we need to check that, for every $N \in \NN$
\begin{equation}\label{eq:SNtildeisEM}
\widetilde{S}^{N,R}_{t} = \frac{1}{N}\sum_{i=1}^{N} \delta_{(\widetilde{X}^{i,N,R}_{t},\widetilde{V}^{i,N,R}_{t})}.		
\end{equation}	
To prove this, consider the functional $\Phi^{S,N}$ defined as:
\begin{multline*}
\Phi^{S,N}(S^{N,R},(X^{i,N,R})_{i\leq N},(V^{i,N,R})_{i\leq N})\\
:= \sup_{\phi \in C_{b}(\TT^{2}\times \RR^{2})}\sup_{t\in[0,T]} \abs{ \left\langle S^{N,R},\phi \right\rangle  -\frac{1}{N}\sum_{i=1}^{N} \phi(X^{i,N,R}_{t},V^{i,N,R}_{t})}
\end{multline*}
which is a measurable functional, and note that this is identically zero by definition of $S^{N,R}$. Now, by the fact that the random vectors\\ $(S^{N,R},(X^{i,N,R})_{i\leq N},(V^{i,N,R})_{i\leq N})$ and $(\widetilde{S}^{N,R},(\widetilde{X}^{i,N,R})_{i\leq N},(\widetilde{V}^{i,N,R})_{i\leq N})$ share the same law, we have
\begin{multline*}
\mathbb{E}^{\widetilde{\PP}}\left[\Phi^{S,N}(\widetilde{S}^{N,R},(\widetilde{X}^{i,N,R})_{i\leq N},(\widetilde{V}^{i,N,R})_{i\leq N})  \right] \\= \mathbb{E}^{{\PP}}\left[\Phi^{S,N}(S^{N,R},(X^{i,N,R})_{i\leq N},(V^{i,N,R})_{i\leq N})\right] = 0.
\end{multline*}
Hence, we conclude that $\Phi^{S,N}(\widetilde{S}^{N,R},(\widetilde{X}^{i,N,R})_{i\leq N},(\widetilde{V}^{i,N,R})_{i\leq N})$ is identically zero $\widetilde{\PP}$-almost surely, which implies \eqref{eq:SNtildeisEM}.
\item To prove that the new object satisfies the same equation  as the initial one, for each $N \in \NN$ we consider a bounded and measurable functional $\Phi^{N}$  taking as argument the function $u^{N,R}$, the particles $(X^{i,N,R})_{i\leq N}$, $(V^{i,N,R})_{i\leq N}$ and the Brownian motions $\left({B}^{i}\right)_{i\leq N}$, that vanishes in expected value  on solutions of system of equation \eqref{eq:PS-VNSR}. The measurability of $\Phi^{N}$ 
follows by the path-by-path formulation while the boundedness requirement can be dealt with by considering a sequence $\Phi^{M,N} := \Phi^{N} \wedge M$ and passing to the limit in $M$ inside the expected value by monotone convergence.
By the equality in law of the new sequences with respect to the initial one, we have that the functional $\Phi^{N}$ will vanish also on the new objects, when averaged with respect to $\widetilde{\PP}$, namely (we omit some technical details of integrability)
\begin{multline*}
\mathbb{E}^{\widetilde{\PP}}\left[\Phi^{N}\bigg(\widetilde{u}^{N,R},(\widetilde{X}^{i,N,R})_{i\leq N},(\widetilde{V}^{i,N,R})_{i\leq N},({\widetilde{B}}^{i,N})_{i\leq N})\right]=\\
\mathbb{E}^{\PP}\left[\Phi^{N}\bigg(u^{N,R},(X^{i,N,R})_{i\leq N},(V^{i,N,R})_{i\leq N}),({B}^{i})_{i\leq N}\bigg)\right]=0.
\end{multline*}
Hence, $(\widetilde{u}^{N,R},(\widetilde{X}^{i,N,R})_{i\leq N},(\widetilde{V}^{i,N,R})_{i\leq N}),({\widetilde{B}}^{i,N})_{i\leq N})$ satisfies system of equation \eqref{eq:PS-VNSR} in the new probability space $(\widetilde{\Omega}, \widetilde{\mathcal{F}},\{\widetilde{\mathcal{F}}_{t}\},\widetilde{\PP})$ which ends this part.

\item Consider now the sequence $({u}^{N},(X^{i,N})_{i\leq N},(V^{i,N})_{i\leq N}))_{N \in \NN}$, associated with system of equations \eqref{eq:PS-VNS}, that is the particle system without the cut-off. 
On the new probability space $(\widetilde{\Omega}, \widetilde{\mathcal{F}},\{\widetilde{\mathcal{F}}_{t}\},\widetilde{\PP})$ consider the same system of equations \eqref{eq:PS-VNS}, i.e. the system of equation where the Brownian motions $(B^{i})_{i\leq N}$ are replaced by the Brownian motions $(\widetilde{B}^{i,N})_{i\leq N}$ introduced in Step 1. Call $(\widetilde{u}^{N},(\widetilde{X}^{i,N})_{i\leq N},(\widetilde{V}^{i,N})_{i\leq N}))_{N\in\NN}$ the solution of such system, which can be seen as a random variable on $(\widetilde{\Omega}, \widetilde{\mathcal{F}},\{\widetilde{\mathcal{F}}_{t}\},\widetilde{\PP})$. Since solutions of system \eqref{eq:PS-VNS} are unique in law we conclude that for all $N \in \NN$
\[
(\widetilde{u}^{N},(\widetilde{X}^{i,N})_{i\leq N},(\widetilde{V}^{i,N})_{i\leq N})) \stackrel{\mathcal{L}aw}{\sim}({u}^{N},(X^{i,N})_{i\leq N},(V^{i,N})_{i\leq N})).
\]
\noindent Also introduce the analogue of the empirical measure $S^{N}$ on the new space
\[
\widetilde{S}^{N}_{t} := \frac{1}{N}\sum_{i=1}^{N} \delta_{(\widetilde{X}^{i,N}_{t},\widetilde{V}^{i,N}_{t})}.
\]
By the previous definition and by construction of 
$((X^{i,N})_{i\leq N},(V^{i,N})_{i\leq N})$ we immediately have 
\[
(u^{N},S^{N})\stackrel{\mathcal{L}aw}{\sim}(\widetilde{u}^{N},\widetilde{S}^{N}),\quad \forall N \in\NN.
\]
\end{enumerate}
\noindent \emph{Step 3:} We can now apply Theorem \ref{teo:GP}. We have to define all the objects needed in the Theorem and verify all the four conditions required. Let $E = C([0,T]\times \TT^{2})\times C([0,T];\PP_{1}(\TT^{2}\times \RR^{2})$ and let $x\in E $ be the couple $(u,F)$. Now we take
\[
X_{N} := (\widetilde{u}^{N,R},\widetilde{S}^{N,R}),\quad Y_{N} := (\widetilde{u}^{N},\widetilde{S}^{N}).
\]
Now define, for $\widetilde{\omega} \in \widetilde{\Omega}$
\begin{multline*}
S_{N}^{R}(\widetilde{\omega}) =  \bigg\{ (w,(x^{i})_{i\leq N},(v^{i})_{i\leq N})\in C([0,T]\times \TT^{2})\times C([0,T];\TT^{2}\times \RR^{2})^{N}\,\text{s.t.}\\ (w,(x^{i})_{i\leq N},(v^{i})_{i\leq N}) \text{ solves } \eqref{eq:PS-VNSR} \text{ with additive noise }(B^{i}_{t}(\widetilde{\omega}))_{i\leq N} \hspace{-0.1cm}\bigg\},
\end{multline*}
the set of path-by-path solutions for system of equations \eqref{eq:PS-VNSR}. We also introduce the analogue for $\eqref{eq:PS-VNS}$: call it $S_{N}(\widetilde{\omega})$.
% \begin{multline*}
% S_{N}(\widetilde{\omega}) =  \bigg\{ (w,(x^{i})_{i\leq N},(v^{i})_{i\leq N})\in C([0,T]\times \TT^{2})\times C([0,T];\TT^{2}\times \RR^{2})^{N}\,\text{s.t.}\\ (w,(x^{i})_{i\leq N},(v^{i})_{i\leq N}) \text{ solves } \eqref{eq:PS-VNS} \text{ with additive noise }(B^{i}_{t}(\widetilde{\omega}))_{i\leq N} \bigg\}.
% \end{multline*}
Now consider 
\begin{multline*}
S_{N}^{X}(\widetilde{\omega})\hspace{-0.1cm} := \hspace{-0.1cm}\bigg\{ (w,\mu)\in E\,|\, \mu=\frac{1}{N}\sum_{i=1}^{N}\delta_{(x^{i},v^{i})}\,,\,(w,(x^{i})_{i\leq N},(v^{i})_{i\leq N})\in S_{N}^{R}(\widetilde{\omega})\bigg\},
\end{multline*}
and
\begin{multline*}
S_{N}^{Y}(\widetilde{\omega})\hspace{-0.1cm} := \hspace{-0.1cm}\bigg\{ (w,\mu)\in E\,|\, \mu=\frac{1}{N}\sum_{i=1}^{N}\delta_{(x^{i},v^{i})}\,,\,(w,(x^{i})_{i\leq N},(v^{i})_{i\leq N})\in S_{N}(\widetilde{\omega})\bigg\}.
\end{multline*}
Roughly speaking, $S_{N}^{X}$ (resp $S_{N}^{Y}$) are the set of couples $(w,\mu)$ where $u$ is a function and $\mu$ is a measure, such that $\mu$ is the empirical measure of a set of particles which, together with $u$, are path-by-path solutions of \eqref{eq:PS-VNSR}. This is just a way of rewriting sets of path-by-path solutions, which match the structure of the space $E$ where the converging objects belong. \\
Now we just need to verify rigorously all the four conditions stated in this Theorem:
\begin{enumerate}
\item In the first Step of this proof, we saw that 
\[
\left(\widetilde{u}^{N,R},\widetilde{S}^{N,R}\right)\xrightarrow{N\rightarrow\infty} \left(u,F\right)\quad \widetilde{\PP}\text{-a.s.}
\]
which correspond exactly to condition 1.  
\item Introduce
\[
\Omega_{B} = \left\{ \widetilde{\omega}\in\widetilde{\Omega}\,|\, (\widetilde{B}^{i,N}(\widetilde{\omega}))_{i\leq N,N\in\NN} \text{ are continuous } \right\}
\]
and note that, since we are considering a countable set of Brownian Motions, this set is of full measure with respect to $\widetilde{\PP}$. Then, by Proposition \ref{prop:pathuniqueness}, we have that 
\[
\sharp S_{N}(\widetilde{\omega})\leq 1 \quad \forall \widetilde{\omega} \in \Omega_{B}.
\]
Hence, the same result holds for the set $S_{N}^{Y}(\widetilde{\omega})$. 
\item Condition 3. state that $(\widetilde{u}^{N,R},\widetilde{S}^{N,R})$ belongs to the set $S_{N}^{X}$ almost surely. However, in Step 2. of this proof we have verified that on $(\widetilde{\Omega}, \widetilde{\mathcal{F}},\{\widetilde{\mathcal{F}}_{t}\},\widetilde{\PP})$\\ $(\widetilde{u}^{N,R},(\widetilde{X}^{i,N,R})_{i\leq N},(\widetilde{V}^{i,N,R})_{i\leq N}),({\widetilde{B}}^{i,N})_{i\leq N})$ satisfied system of equation \eqref{eq:PS-VNSR} in the sense of SDEs. This condition implies that for fixed $\widetilde{\omega}\in\widetilde{\Omega}$ the vector 
$(\widetilde{u}^{N,R}(\widetilde{\omega}),(\widetilde{X}^{i,N,R}(\widetilde{\omega}))_{i\leq N},(\widetilde{V}^{i,N,R}(\widetilde{\omega}))_{i\leq N})\in S_{N}^{R}(\widetilde{\omega})$. Since in Step 2. we verified that $\widetilde{S}^{N,R}$ is in fact an empirical measures on particle solutions of \eqref{eq:PS-VNSR} and by the definition of $S_{N}^{X}$, this imply the first part of condition 3. The same result holds for $(\widetilde{u}^{N},\widetilde{S}^{N})$ and $S_{N}^{Y}$ by an analogous argument.
\item Condition 4. is the most delicate. Take a couple $(w,\mu) \in S^{X}_{N}(\widetilde{\omega})\cap B_{E}((u,F),1)$. Since $(w,\mu)\in B_{E}((u,F),1) $ we have that
\[
\norm{w}_{C([0,T]\times \TT^{2})} \leq \norm{u}_{C([0,T]\times \TT^{2})}+1.
\]
The couple $(w,\mu)$ also lies in $S^{X}_{N}(\widetilde{\omega})$, hence there exist $((x^{i}),(v^{i}))_{i\leq N}\in C([0,T];\TT^{2}\times \RR^{2})$ such that $(w,(x^{i})_{i\leq N},(v^{i})_{i\leq N}) \in S_{N}^{R}(\widetilde{\omega})$, which means that is a path-by-bath solutions of \eqref{eq:PS-VNSR}. However, since $w$ is uniformly bounded by $\norm{u}_{C([0,T]\times \TT^{2})}+1$, which correspond exactly to our choice of $R$ (see at the beginning of this section), we see that the cut off function 
$\chi_{R}(w) \equiv 1$ is identically one. Hence system of equation \eqref{eq:PS-VNSR} reduce to \eqref{eq:PS-VNS}, which is the particle system without the cut-off. This implies that $(w,(x^{i})_{i\leq N},(v^{i})_{i\leq N})$ solves also \eqref{eq:PS-VNS}, hence $(w,\mu) \in S^{Y}_{N}(\widetilde{\omega})$. 
\end{enumerate}
Since we verified all the necessary conditions, we can apply Theorem \ref{teo:GP}, obtaining 
\[
\left(\widetilde{u}^{N},\widetilde{S}^{N}\right)\xrightarrow{N\rightarrow\infty} \left(u,F\right)\quad \widetilde{\PP}\text{-a.s.}
\]
Since almost sure convergence implies convergence in law, and since we verified in Step 2. that 
\[
(u^{N},S^{N})\stackrel{\mathcal{L}aw}{\sim}(\widetilde{u}^{N},\widetilde{S}^{N}),\quad \forall N \in\NN.
\]
we can transport this kind of convergence to the original probability space $(\Omega, \mathcal{F},\left\{\mathcal{F}_{t}\right\},\PP)$, hence ending the proof. 
\end{proof}

\newpage
% \printbibliography
% \input{papero.bbl}

\bibliography{VNSbiblio}{}

\begin{thebibliography}{10}

\bibitem{allaire1991homogenizationI}
Gr{\'e}goire Allaire.
\newblock Homogenization of the navier-stokes equations in open sets perforated
  with tiny holes {I}. abstract framework, a volume distribution of holes.
\newblock {\em Archive for Rational Mechanics and Analysis}, 113(3):209--259,
  1991.

\bibitem{allaire1991homogenizationII}
Gr{\'e}goire Allaire.
\newblock Homogenization of the navier-stokes equations in open sets perforated
  with tiny holes {II}: Non-critical sizes of the holes for a volume
  distribution and a surface distribution of holes.
\newblock {\em Archive for Rational Mechanics and Analysis}, 113(3):261--298,
  1991.

\bibitem{golse2016aerosol}
Etienne Bernard, Laurent Desvillettes, Fran{\c{c}}ois Golse, and Valeria Ricci.
\newblock A derivation of the vlasov-navier-stokes model for aerosol flows from
  kinetic theory.
\newblock {\em arXiv:1608.00422}, 2016.

\bibitem{billingsley2013convergence}
Patrick Billingsley.
\newblock {\em Convergence of probability measures}.
\newblock John Wiley \& Sons, 2013.

\bibitem{boudin2009global}
Laurent Boudin, Laurent Desvillettes, C{\'e}line Grandmont, and Ayman Moussa.
\newblock Global existence of solutions for the coupled vlasov and
  navier-stokes equations.
\newblock {\em Differential and integral equations}, 22(11/12):1247--1271,
  2009.

\bibitem{3luo}
Myeongju Chae, Kyungkeun Kang, and Jihoon Lee.
\newblock Global existence of weak and classical solutions for the
  navier-stokes-vlasov-fokker-planck equations.
\newblock {\em Journal of Differential Equations}, 251(9):2431--2465, 2011.

\bibitem{davie2007}
Alexander~M Davie.
\newblock Uniqueness of solutions of stochastic differential equations.
\newblock {\em International Mathematics Research Notices}, 2007, 2007.

\bibitem{degondMax}
Pierre Degond.
\newblock Global existence of smooth solutions for the vlasov-fokker-planck
  equation in 1 and 2 space dimensions.
\newblock In {\em Annales scientifiques de l'{\'E}cole Normale Sup{\'e}rieure},
  volume~19, pages 519--542. Elsevier, 1986.

\bibitem{golse2008meanfield}
Laurent Desvillettes, Fran{\c{c}}ois Golse, and Valeria Ricci.
\newblock The mean-field limit for solid particles in a navier-stokes flow.
\newblock {\em Journal of Statistical Physics}, 131(5):941--967, 2008.

\bibitem{mathiaud2010}
Laurent Desvillettes and Julien Mathiaud.
\newblock Some aspects of the asymptotics leading from gas-particles equations
  towards multiphase flows equations.
\newblock {\em Journal of Statistical Physics}, 141(1):120--141, 2010.

\bibitem{feireisl2016homogenization}
Eduard Feireisl, Yuliya Namlyeyeva, and {\v{S}}{\'a}rka Ne{\v{c}}asov{\'a}.
\newblock Homogenization of the evolutionary navier--stokes system.
\newblock {\em Manuscripta Mathematica}, 149(1-2):251--274, 2016.

\bibitem{flandoli2011random}
Franco Flandoli.
\newblock {\em Random Perturbation of PDEs and Fluid Dynamic Models: {\'E}cole
  d’{\'e}t{\'e} de Probabilit{\'e}s de Saint-Flour XL--2010}, volume 2015.
\newblock Springer Science \& Business Media, 2011.

\bibitem{flandoli2016fluid}
Franco Flandoli.
\newblock A fluid-particle system related to vlasov-navier-stokes equations.
\newblock {\em Mathematical Analysis of Viscous Incompressible Fluid}, RIMS
  K{\^o}ky{\^u}roku 2058, Y. Maekawa Ed., RIMS, Kyoto, 2017.

\bibitem{oldVNS}
Franco Flandoli, Marta Leocata, and Cristiano Ricci.
\newblock The vlasov-navier-stokes equations as a mean field limit.
\newblock {\em Discrete and Continuous Dynamical Systems - B}, 22, 2017.

\bibitem{flandoli2012stochastic}
Franco Flandoli and Alex Mahalov.
\newblock Stochastic three-dimensional rotating navier--stokes equations:
  averaging, convergence and regularity.
\newblock {\em Archive for Rational Mechanics and Analysis}, 205(1):195--237,
  2012.

\bibitem{Sueur2018point}
Olivier Glass, Alexandre Munnier, and Franck Sueur.
\newblock Point vortex dynamics as zero-radius limit of the motion of a rigid
  body in an irrotational fluid.
\newblock {\em Inventiones mathematicae}, 214(1):171--287, 2018.

\bibitem{goudon2010navier}
Thierry Goudon, Lingbing He, Ayman Moussa, and Ping Zhang.
\newblock The navier--stokes--vlasov--fokker--planck system near equilibrium.
\newblock {\em SIAM Journal on Mathematical Analysis}, 42(5):2177--2202, 2010.

\bibitem{goudon2004light}
Thierry Goudon, Pierre-Emmanuel Jabin, and Alexis Vasseur.
\newblock Hydrodynamic limit for the vlasov-navier-stokes equations. part {I}:
  Light particles regime.
\newblock {\em Indiana University mathematics journal}, pages 1495--1515, 2004.

\bibitem{goudon2004fine}
Thierry Goudon, Pierre-Emmanuel Jabin, and Alexis Vasseur.
\newblock Hydrodynamic limit for the vlasov-navier-stokes equations. part {II}:
  Fine particles regime.
\newblock {\em Indiana University mathematics journal}, pages 1517--1536, 2004.

\bibitem{Hamdache1998}
Kamal Hamdache.
\newblock Global existence and large time behaviour of solutions for the
  vlasov-stokes equations.
\newblock {\em Japan journal of industrial and applied mathematics}, 15(1):51,
  1998.

\bibitem{han2017uniqueness}
Daniel Han-Kwan, {\'E}velyne Miot, Ayman Moussa, and Iv{\'a}n Moyano.
\newblock Uniqueness of the solution to the 2d vlasov-navier-stokes system.
\newblock {\em arXiv:1710.07427}, 2017.

\bibitem{otto2004}
Pierre-Emmanuel Jabin and Felix Otto.
\newblock Identification of the dilute regime in particle sedimentation.
\newblock {\em Communications in mathematical physics}, 250(2):415--432, 2004.

\bibitem{oelschlager1985law}
Karl Oelschl{\"a}ger.
\newblock A law of large numbers for moderately interacting diffusion
  processes.
\newblock {\em Zeitschrift f{\"u}r Wahrscheinlichkeitstheorie und verwandte
  Gebiete}, 69(2):279--322, 1985.

\bibitem{simon1986compact}
Jacques Simon.
\newblock Compact sets in the space {Lp(O,T;B)}.
\newblock {\em Annali di Matematica pura ed applicata}, 146(1):65--96, 1986.

\bibitem{sznitman1991topics}
Alain-Sol Sznitman.
\newblock Topics in propagation of chaos.
\newblock In {\em Ecole d'{\'e}t{\'e} de probabilit{\'e}s de Saint-Flour
  XIX—1989}, pages 165--251. Springer, 1991.

\bibitem{CYu}
Cheng Yu.
\newblock Global weak solutions to the incompressible navier--stokes--vlasov
  equations.
\newblock {\em Journal de Math{\'e}matiques Pures et Appliqu{\'e}es},
  100(2):275--293, 2013.

\end{thebibliography}
\bibliographystyle{plain}

\end{document}